\documentclass[reqno]{amsart}
\usepackage{amsfonts,amssymb,amsmath,latexsym,indentfirst,cite}
\usepackage{amsthm}
\usepackage{color,graphicx}
\usepackage{eucal}
\newcommand{\R}{\mathbb R}

\newcommand{\Z}{\mathbb Z}

\newcommand{\N}{\mathbb N}

\def\Re{ \mathrm{Re}\, }

\newcommand{\eps}{\epsilon}

\newcommand{\peq}{\hspace*{0.10in}}
\newcommand{\peqq}{\hspace*{0.05in}}

\newcommand{\ra}{\rightarrow}

\newtheorem{theorem}{Theorem}[section]
\newtheorem{proposition}[theorem]{Proposition}
\newtheorem{remark}[theorem]{Remark}
\newtheorem{lemma}[theorem]{Lemma}

\newtheorem{definition}[theorem]{Definition}

\begin{document}
\vglue-1cm \hskip1cm
\title[lack of compactness and existence of maximizers]{On the lack of compactness and existence of maximizers for some Airy-Strichartz inequalities}

%Linear Profile decomposition for the Airy equation in Sobolev spaces and the existence of maximizers for Strichartz inequalities

%\author[]{}
%\address{}
%\email{}
%\thanks{}
%\begin{abstract}
%\end{abstract}

\begin{abstract}
This work is devoted to prove a linear profile decomposition for the Airy equation in $\dot{H}_x^{s_k}(\R)$, where $s_k=(k-4)/2k$ and $k>4$. We also apply this decomposition to establish the existence of maximizers for a general class of Strichartz type inequalities associated to the Airy equation.

%We prove that any sequence of solutions to the linear Airy equation bounded in $\dot{H}_x^{s_k}(\R)$, where $s_k=(k-4)/2k$ and $k>4$, can be decompose . 

%Our goal here is to prove the existence of maximizers for the above inequality.  The main tool we use is a linear profile decomposition for the Airy equation with initial data in $\dot{H}_x^{s_k}(\R)$. As a consequence, we also establish the existence of maximizers for a more general class of Strichartz type inequalities associated to the Airy equation.

\end{abstract}

%\maketitle

\author[L. G. Farah]{Luiz G. Farah}
\address{ICEx, Universidade Federal de Minas Gerais, Av. Ant\^onio Carlos, 6627, Caixa Postal 702, 30123-970,
Belo Horizonte-MG, Brazil}
\email{lgfarah@gmail.com}

\author[H. Versieux]{Henrique Versieux}
\address{Instituto de Matem\'atica, Universidade Federal do Rio de Janeiro, Av. Athos da Silveira Ramos 149 - Ilha do Fundão. Caixa Postal 68530, 21941-909, Rio de Janeiro-RJ, Brazil}
\email{henrique@im.ufrj.br}

%\thanks{L. G. F. was partially supported by CNPq/Brazil and FAPEMIG/Brazil.}

\maketitle

\section{Profile Decomposition}

In this paper, we investigate the defect of compactness in some space-time estimates enjoyed by the solutions of the Airy equation

\begin{equation}\label{Lkdv}
\begin{cases}
\partial_t u+\partial_x^3u=0, \;\;x\in\R, \;t>0, \\
u(x,0)=u_0(x).
\end{cases}
\end{equation}

We denote the unique solution of \eqref{Lkdv} via the propagator $u(x,t)=U(t)u_0(x)$. Notice that the solution is globally defined in the homogeneous Sobolev space $\dot{H}^s(\R)$, for all $s\in \R$. Moreover, $\{U(t)\}_{t\in \R}$ defines a unitary operator in these homogeneous Sobolev spaces. In particular, we have for all $s\in \R$
\begin{equation}\label{HSP}
\|U(t)u_0\|_{\dot{H}_x^s}=\|u_0\|_{\dot{H}_x^s}, \peq \textrm{ for all } t\in \R.
\end{equation}

Mixed Lebesgue norms $L^p_tL^q_x$ estimates, so called the Strichartz estimates, are the fundamental tools to obtain well-posedness for  nonlinear dispersive problems, such as the
 supercritical generalized Korteweg-de Vries (gKdV) equation
\begin{equation}\label{gkdv}
\begin{cases}
\partial_t u+\partial_x^3u+\partial_x(u^{k+1})=0, \;\;x\in\R, \;t>0, \;k\geq 4\\
u(x,0)=u_0(x).
\end{cases}
\end{equation}
Studying oscillatory integrals and their relationship with smoothing properties of dispersive equations, Kenig, Ponce and Vega\cite[Theorem 2.1]{KPV4} proved, among other general results, that the solution of the linear problem \eqref{Lkdv} satisfies the following estimate
\begin{equation}\label{reg1}
 \|D^{\alpha}_xU(t)u_0\|_{L_t^pL_x^q}\leq C_{p,q,\alpha} \|u_0\|_{L^2_x},
\end{equation}
if the condition
\begin{equation*}
-\alpha+\dfrac{1}{q}+\dfrac{3}{p}=\dfrac{1}{2}, \quad -\dfrac{1}{2}\leq\alpha\leq\dfrac{1}{p}
\end{equation*}
is verified. Continuing their study on the well-posedness of the gKdV equation \eqref{gkdv}, in the seminal work \cite{KPV6} they obtained  another space-time mixed estimate interchanging the position of the variables $x$, $t$ (see also \cite[Theorem 2.1]{KPV2000}). Indeed, assuming that
\begin{equation*}
\frac{1}{q}+\frac{1}{2p}=\frac{1}{4}, \quad \alpha=\frac{2}{p}-\frac{1}{q},
\quad 1\leq p,q\leq\infty, \quad -\frac{1}{4}\leq\alpha\leq1
\end{equation*}
for any $u_0 \in L^2(\R)$ there exists a constant $C_{p,q,\alpha}>0$ such that
\begin{equation*}
\|D^{\alpha}_x U(t)u_0\|_{L^{q}_xL^{p}_t}\leq C_{p,q,\alpha}
\|u_0\|_{L^2_x}.
\end{equation*}
In the particular case where $p=\infty$, it is also proved (see \cite[Lemma 3.29]{KPV6})
\begin{equation}\label{eq41}
\|D^{-\alpha}_x U(t)u_0\|_{L^{q}_xL^{\infty}_t}\leq C_{q,\alpha}
\|u_0\|_{L^2_x}
\end{equation}
for 
\begin{equation*}
\frac{1}{4}\leq \alpha <\frac{1}{2} \quad \textrm{ and } \quad \frac{1}{q}=\frac{1}{2}-\alpha.
\end{equation*}

All the above estimates play a fundamental role in the development of the corresponding local and global theory for the gKdV equation \eqref{gkdv}. Indeed, small data global well-posedness in the critical homogeneous Sobolev space $s_k=(k-4)/2k$, with $k\geq 4$ was established in \cite{KPV6}.

Notice that the Strichartz estimates \eqref{reg1} can also be rewritten in  terms of the critical homogeneous Sobolev space $\dot{H}_x^{s_k}(\R)$. Indeed, let $k\geq 4$ and assume $u_0 \in \dot{H}_x^{s_k}(\R)$ then there exists a constant $C_{p,q,k}>0$ such that
\begin{equation}\label{STR0}
 \|D^{1/p}_xU(t)u_0\|_{L_t^pL_x^q}\leq C_{p,q,k}\|u_0\|_{\dot{H}_x^{s_k}},
\end{equation}
where
\begin{equation}\label{Adm}
\dfrac{2}{p}+\dfrac{1}{q}=\dfrac{2}{k}.
\end{equation}
We say that a pair $(p,q)$ is $\dot{H}_x^{s_k}$-admissible if the last condition is satisfied. Typically, the particular case where $p=q$ is very useful in the  study of nonlinear  dispersive equations. For future reference we state this case below
\begin{equation}\label{STR}
 \|D^{2/3k}_xU(t)u_0\|_{L_{x,t}^{3k/2}}\leq C_{k} \|u_0\|_{\dot{H}_x^{s_k}}.
\end{equation}
Throughout the paper, other two particular cases of inequality \eqref{STR0} are of special interest
\begin{equation}\label{STR1}
p=\infty \,\,\textrm{ and }\,\, q=k/2
\end{equation}
and
\begin{equation}\label{STR2}
p=k+2 \,\,\textrm{ and }\,\, q=k(k+2)/4.
\end{equation}
It should be pointed out that the case \eqref{STR1} is, in fact, a consequence of the $\dot{H}_x^s(\R)$-norm preservation \eqref{HSP} and Sobolev embedding. On the other hand, for $k\geq 4$, we have $p<q$ in \eqref{STR2} and since $3k/2>k+2$ the inequality \eqref{STR} can also be obtained by complex interpolation between \eqref{STR1} and \eqref{STR2}. Also note that for $k=4$ estimates \eqref{STR} and \eqref{STR2} are the same.

Recently, the first author in a joint work with Ademir Pastor \cite{FP13} presented a simpler proof of the classical well-posedness result for the generalized KdV equation \cite{KPV6}. The key ingredient in the proof is the following Airy-Strichartz estimate

\begin{lemma}\label{lemma12} Let $k>4$, $s_k=(k-4)/2k$.
Then
\begin{equation}\label{STR3}
 \|U(t)u_0\|_{L^{5k/4}_xL^{5k/2}_t}\leq C_k\|u_0\|_{\dot{H}_x^{s_k}}
\end{equation}
\end{lemma}
\begin{proof}
See \cite[Lemma 2.5]{FP13}.
\end{proof}

The crucial point here, due to the invariance of the Airy equation \eqref{Lkdv} under some transformations, is that the Strichartz inequalities \eqref{STR0} and \eqref{STR3} cannot guarantee the solution map from $\dot{H}_x^s(\R)$ to the Strichartz space to be compact. Indeed, it is easy to see that the norm in $\dot{H}_x^{s_k}(\R)$ is invariant under translations $\tau_y$, dilatations $\delta_h$ and frequency modulations $R_{t_0}$ defined respectively by
\begin{equation}\label{INV}
\tau_y\varphi(x)=\varphi(x-y), \quad \delta_h\varphi(x)=\frac{1}{h^{2/k}}\varphi\left(\frac{x}{h}\right) \quad \textrm{and}\quad R_{t_0}\varphi(x)=U(t_0)\varphi(x).
\end{equation}
Moreover, let $(x_n)_{n\in \mathbb{N}}$ and $(t_n)_{n\in \mathbb{N}}$ be sequences in $\R$ both going to infinity and $(h_n)_{n\in \mathbb{N}}$ be a sequence in $\R_+^{\ast}$ going to zero. If $\varphi$ is a non zero fixed element of $\dot{H}_x^{s_k}(\R)$ then it is clear that $(\tau_{x_n}\varphi)_{n\in \mathbb{N}}$, $(\delta_{h_n}\varphi)_{n\in \mathbb{N}}$ and $(R_{t_n}\varphi)_{n\in \mathbb{N}}$ converge weakly to $0$. However, for every $\dot{H}_x^{s_k}$-admissible pair $(p,q)$ (see \eqref{Adm}) the Strichartz norms of these sequences are all equal to $\|D^{1/p}_xU(t)\varphi\|_{L_t^pL_x^q}$. The same happens with the Strichartz norm $L_x^{5k/4}L^{5k/2}_t$. Therefore, we cannot expect the solution map of equation \eqref{Lkdv} from the space $\dot{H}_x^{s_k}(\R)$ to the Strichartz spaces to be compact and indeed the above sequences are not relatively compact in the Strichartz spaces.

The failure of compactness also appears in the Sobolev embedding $\dot{H}^s(\R^d)\subset L^{\frac{2d}{d-2s}}$ and was clarified by G\'erard \cite{Ge98}. The same problem in the context of dispersive equations have received a great deal of attention since the last decade. Inspired by \cite{Ge98}, Bahouri and G\'erard \cite{BaGe99} and Keraani \cite{Ke01} obtained analogous results related to the Sobolev-Strichartz estimates, respectively for the wave and Schr\"odinger equations. Roughly speaking, they proved that a sequence of solutions to the linear equation with bounded data in a certain homogeneous Sobolev space can be decomposed, up to a subsequence, into a sum of \emph{almost orthogonal} concentrating waves and a rest which is small in the associated Strichartz norms. In the literature, this is referred as \emph{linear profile decomposition}. We should also mention the works of Merle and Vega \cite{MeVe98} and B\'egout and Vargas \cite{BeVa07} for the Schr\"odinger equation, Ramos \cite{Ra12} for the wave equation and Shao \cite{Sh092} for the critical KdV equation (\eqref{gkdv} with $k=4$), where Strichartz estimates at the lowest scales are considered and some refined Strichartz inequalities are needed. 

In the spirit of G\'erard \cite{Ge98} (see also \cite{Ke01, BaGe99, Sh092}) we prove that the transformations $\tau_y$, $\delta_h$ and $R_{t_0}$ defined in \eqref{INV} are the only responsible for the lack of compactness in the Airy-Strichartz estimates \eqref{STR3} and \eqref{STR0}, with $(p,q)$ satisfying \eqref{Adm} and $k+2<p<\infty$. More precisely, we prove the following linear profile decomposition result for the Airy equation in the critical homogeneous Sobolev space $\dot{H}_x^{s_k}(\R)$.

\begin{theorem}\label{profdec}
 Let $k>4$ and ${\bf u}=(u_n)_{n\in \mathbb{N}}$ be a bounded sequence in $\dot{H}_x^{s_k}(\R)$, with $s_k=(k-4)/2k$. Set $v_n=U(t)u_n$. Then there exists a subsequence, which we still denote by ${\bf v}=(v_n)_{n\in \mathbb{N}}$, a sequence of functions $(\psi^j)_{j\in \mathbb{N}}\subset \dot{H}_x^{s_k}(\R)$ and a family of parameters $(h_n^j, x_n^j, t_n^j)_{n\in \mathbb{N}, j\in \mathbb{N}}$  such that for every $J\geq1$ there exists $R^J_n\subset \dot{H}_x^{s_k}(\R)$ satisfying
\begin{equation*}
v_n(x,t)=\sum_{j=1}^{J}\dfrac{1}{(h_n^j)^{2/k}}U\left(\dfrac{t-t_n^j}{(h_n^j)^{3}}\right)\psi^j\left(\dfrac{x-x_n^j}{h_n^j}\right) + U(t)R_n^J(x),
\end{equation*}
where the reminder sequences  $(R_n^J)_{n\in\N}$ satisfies  the following asymptotic smallness property
\begin{equation}\label{SPWNL}
\limsup_{n\rightarrow \infty} \|D_x^{1/p}U(t)R_n^J\|_{L^p_tL^{q}_x} \rightarrow 0,  \peqq \textrm{as} \peqq J\rightarrow \infty,
\end{equation}
for all $\dot{H}_x^{s_k}$-admissible pair $(p,q)$, with $k+2<p<\infty$, satisfying \eqref{Adm}.

Furthermore, the sequences of parameters have a pairwise divergence property: For all $1\leq i \neq j\leq J$,
\begin{equation}\label{XT}
\lim_{n\rightarrow \infty} \left|\dfrac{h_n^i}{h_n^j}\right|+\left|\dfrac{h_n^j}{h_n^i}\right|+\left|\dfrac{t_n^i-t_n^j}{(h_n^i)^{3}}\right|+\left|\dfrac{x_n^i-x_n^j}{h_n^i}\right|=\infty.
\end{equation}

Finally, for fixed $J\geq 1$, we have the asymptotic Pythagorean expansion
\begin{equation}\label{PYTHA}
\|v_n(\cdot,0)\|^2_{\dot{H}_x^{s_k}}-\sum_{j=1}^{J}\|\psi^j\|^2_{\dot{H}_x^{s_k}} - \|R_n^J\|^2_{\dot{H}_x^{s_k}} =o(1),  \peqq \textrm{as} \peqq n\rightarrow \infty.
\end{equation}
\end{theorem}

\begin{remark}\label{R42}
By analytic interpolation and inequalities \eqref{STR}, \eqref{eq41} we have
\begin{equation*}
\begin{split}
\|U(t)u_0\|_{L^{5k/4}_xL^{5k/2}_t} &\leq c \|D_x^{-1/k}U(t)u_0\|_{L^{k}_xL^{\infty}_t}^{2/5} \|D_x^{2/3k}U(t)u_0\|_{L^{3k/2}_{x,t}}^{3/5}\\
&\leq c \|D_x^{s_k}u_0\|_{L^2_x}^{2/5} \|D_x^{2/3k}U(t)u_0\|_{L^{3k/2}_{x,t}}^{3/5} .
\end{split}
\end{equation*}
Since $3k/2>k+2$, the relations \eqref{SPWNL} and \eqref{PYTHA} imply that the remainder also asymptotically vanishes in the Strichartz space $L_x^{5k/4}L^{5k/2}_t$ (when $k=4$ this was already observed by Killip, Kwon, Shao and Visan \cite{KKSV12}).
\end{remark}

Our proof of Theorem \ref{profdec} is based on the classical linear profile decomposition for the energy critical Schr\"odinger equation proved by Keraani \cite{Ke01} and also on the the ideas recently developed by Fanelli and Visciglia \cite{FV13} in a general setting. We strongly used the hypothesis  $s_k>0$, which holds for $k>4$, and as consequence our arguments cannot be adapted for the $L^2$-critical case $k=4$ (note that $s_4=0$). As mention above, this case was already considered by Shao \cite{Sh092} and the main tool in his proof is a refined version of the Airy-Strichartz estimate \eqref{STR} with $k=4$. Here, we replace the Strichartz refinement by the inequality \eqref{Gereq} below proved by G\'erard \cite{Ge98}.

The profile decompositions has proven to be a powerful tool to study several different problems related to nonlinear dispersive equations. For instance, they can be useful in establishing the global well-posedness and scattering results for the Schr\"odinger and wave equations at critical regularity, see \cite{KeMe06, KeMe08, TaViZh07, KTV09, KV10}. Another important applications was in the study of mass concentration phenomena near the blow up time for the mass critical Schr\"odinger equation, see \cite{MeVe98, BeVa07, Ke06}. It was also used to show the existence of maximizers for Sobolev-Strichartz inequalities related to the Schr\"odinger, wave and KdV equations, see for instance \cite{Ku03 , Fo07, Bu10, Sh091, Sh092} and references therein.

In the present work, as an application of the linear profile decomposition stated in Theorem \ref{profdec}, we prove the existence of maximizers for inequalities \eqref{STR0} and \eqref{STR3}. This problem has attracted a lot of attention in the last few years for other Strichartz  estimates, see for instance \cite{Ku03 , Fo07, Ca09, Bu10, Sh091, Sh092, FVV12, BeRo13}. Existence of maximizers for Strichartz inequalities associated with the Schr\"odinger equation in one dimension were first proved by Kunze \cite{Ku03}. Later, Foschi \cite{Fo07} showed the existence and uniqueness of maximizers for the Schr\"odinger propagator, and their exactly Gaussian format. Similar results were obtained through a different strategy by Hundertmark and Zharnitsky \cite{hum-zhar}; see also Carneiro \cite{Ca09} for a generalization of their results. In the nonlinear setting, Duyckaerts, Merle and Roudenko\cite{Du-Me-Ro} proved that the maximum of some $L^p$ norm in the time and space variables (related to the linear problem) is attained for a given small mass solution of the $L^2$-critical nonlinear Schr\"odinger equation in general dimension. Moreover, in dimension one and two, they showed that the maximizer is unique and also obtained a precise estimate of the maximum. This was the first result showing the existence of a maximizer considering a nonlinear dispersive equation. For the wave equation, Bez and Rogers \cite{BeRo13} succeed in characterizing the best constant of the  Strichartz inequality and also the shape of maximizers in dimension five. The shape of the maximizers remains an open problem for all Airy-Strichartz inequalities and we only consider the existence of maximizers in the present work. Our first result in this direction is related to the Airy-Strichartz inequality \eqref{STR0}.
\begin{theorem}\label{Max1}
Let $k>4$ and $(p,q)$ be an $\dot{H}_x^{s_k}$-admissible pair $(p,q)$ satisfying \eqref{Adm}, with $k+2<p<\infty$. Then there exists a maximizing function $\psi_k\in \dot{H}_x^{s_k}$ such that
$$ 
\|D^{1/p}_xU(t)\psi_k\|_{L_t^pL_x^q}=M_k\|D_x^{s_k}\psi_k\|_{L^2_x},
$$
with
$$
M_k=\sup \{ \|D^{1/p}_xU(t)\psi\|_{L_t^pL_x^q} : \psi\in \dot{H}_x^{s_k}, \|D_x^{s_k}\psi\|_{L^2_x}=1\}
$$
being the sharp constant.
\end{theorem}

It's worth mentioning that the above theorem also guarantee the existence of maximizers for the Airy-Strichartz inequality \eqref{STR}. This inequality in the critical case $k=4$ was studied by Shao \cite{Sh092}. His work is based on the asymptotic embedding of a NLS solution to an approximate critical gKdV solution obtaining a dichotomy result. The main result in \cite{Sh092} states that  there exists either a maximizer for \eqref{STR} (with $k=4$) or a function  $\phi \in L_x^2$  with $\|\phi\|_{L^2}=1$ and a sequence $(a_n)_{n\in \N}$  with $|a_n| \rightarrow \infty$ such that the sequence of translations $e^{i (\cdot)a_n } \phi$ has Strichartz norm converging to the sharp constant. Our result in Theorem \ref{Max1} can be viewed as an extension of Shao's result since it removes the dichotomy, when $k>4$, and give a positive answer for the existence of maximizers also for a general class of Strichartz type inequalities associated to the Airy equation.

The next result is concerned with the Airy-Strichartz inequality \eqref{STR3}. To the best of our knowledge,  this  is the first result  in the literature about the existence of maximizers for Strichartz estimates interchanging the position of the variables $x$, $t$.
\begin{theorem}\label{Max2}
Let $k>4$, then there exists a maximizing function $\Psi_k\in \dot{H}_x^{s_k}$ such that
\begin{equation}\label{STRM2}
 \|U(t)\Psi_k\|_{L^{5k/4}_xL^{5k/2}_t}=L_k\|D_x^{s_k}\Psi_k\|_{L^2_x},
\end{equation}
with
$$
L_k=\sup \{ \|U(t)\Psi\|_{L^{5k/4}_xL^{5k/2}_t} : \Psi\in \dot{H}_x^{s_k}, \|D_x^{s_k}\Psi\|_{L^2_x}=1\}
$$
being the sharp constant.
\end{theorem}

It should be pointed out that existence of maximizers for Sobolev-Strichartz inequalities associated to a general class of propagators were studied by Fanelli, Vega and Visciglia \cite{FVV12}. In this paper, the authors give a unified proof which involves a large class of examples of propagators. Indeed, the authors proved the existence of maximizers for Sobolev-Strichartz estimates from homogeneous Sobolev space $\dot{H}^s(\R)$, with $s>0$, to spaces $L^r_{x,t}$. Our result in Theorem \ref{Max1} is more general since we obtain the existence of maximizers for anisotropic Strichartz estimates in spaces $L^p_tL^q_x$, for general couples $(p,q)$ satisfying \eqref{Adm} with $k+2<p<\infty$. Moreover, Theorem \ref{Max2} also deals with the anisotropic Strichartz space $L_x^{5k/4}L^{5k/2}_t$.

This paper is organized as follows. In the next section we introduce some notation and prove the linear profile decomposition result stated in Theorem \ref{profdec}. Next, in section \ref{mainresult}, we present the proof Theorems \ref{Max1} and \ref{Max2}.

\section{Profile Decomposition in $\dot{H}_x^{s_k}(\R)$}

We start this section by introducing the notation used throughout the paper. We use $c$ to denote various constants that may vary line by line. Given any positive numbers $a$ and $b$, the notation $a \lesssim b$ means that there exists a positive constant $c$ such that $a \leq cb$, with $c$ uniform with respect to the set where $a$ and $b$ vary. We also use the notation $a_n=o(1)$ as $n\rightarrow \infty$ to denote the limit $\lim_{n\rightarrow \infty} a_n=0$.

The $L^p(\R)$ norm is denoted by $\|\cdot\|_{L^p}$. If necessary, we use subscript to inform which variable we are concerned with. The mixed norms $L^p_tL^q_x$, $L^q_xL^p_t$ are defined, respectively, as
\begin{equation*}
\|f\|_{L^p_tL^q_x}= \left(\int_{\R}\|f(\cdot,t)\|_{L^q_x}^p dt \right)^{1/p} \textrm{ and } \quad \|f\|_{L^q_xL^p_t}= \left(\int_{\R}\|f(x,\cdot)\|_{L^p_t}^q dx \right)^{1/q}.
\end{equation*}
with the usual modifications when $p =\infty$ or $q=\infty$.

The spatial Fourier transform of $f(x)$ is given by
\begin{equation*}
\hat{f}(\xi)=\int_{\R}e^{-ix\xi}f(x)dx.
\end{equation*}
For a general function $\sigma$ in $\R$, we define the Fourier multiplier $\sigma(D)$ by
$$
\widehat{(\sigma(D)f)}{\xi}=\sigma(\xi)\widehat{f}(\xi).
$$

We shall also define $D_x^s$ to be the Fourier multiplier with symbol $|\xi|^s$. In this case, the norm in the homogeneous Sobolev space $\dot{H}^s(\R)$ is given by
\begin{equation*}
\|f\|_{\dot{H}_x^s}\equiv \|D_x^sf\|_{L^2_x}=\||\xi|^s\hat{f}\|_{L^2_{\xi}}.
\end{equation*}

Next, we recall some notations and results introduced by G\'erard \cite{Ge98}.

\begin{definition}
Let ${\bf f}=(f_n)_{n\in\N}$ be a bounded sequence in $L^2(\R)$ and ${\bf h}=(h_n)_{n\in\N}$, ${\bf \widetilde{h}}=(\widetilde{h}_n)_{n\in\N}$ two sequence of positive scales. We say that
\begin{itemize}
\item[(i)] ${\bf f}$ is ${\bf h}$-oscillatory if
\begin{equation}\label{ho1}
\limsup_{n\rightarrow \infty}
\left(
\int_{h_n|\xi|\leq \frac{1}{R}}|\widehat{f_n}(\xi)|^2d \xi +
\int_{h_n|\xi|\geq {R}}|\widehat{f_n}(\xi)|^2d \xi
\right) \rightarrow 0 \,\,\, \textrm{ as }\,\,\, R\rightarrow \infty.
\end{equation}
\item[(ii)] ${\bf f}$ is ${\bf h}$-singular if, for every $b>a>0$, we have
\begin{equation}\label{ho2}
\lim_{n\rightarrow \infty}
\int_{a\leq h_n|\xi|\leq b}|\widehat{f_n}(\xi)|^2d \xi =0.
\end{equation}
\item[(iii)] ${\bf h}$ and ${\bf \widetilde{h}}$ are orthogonal if
\begin{equation}\label{ho3}
\lim_{n\rightarrow \infty}
\left( \left| \frac{h_n}{\widetilde{h}_n} \right| + \left| \frac{\widetilde{h}_n}{h_n} \right|
\right) =+\infty.
\end{equation}
\end{itemize}
\end{definition}

\begin{remark}\label{Orth}
Note that if ${\bf f}$ is ${\bf h}$-oscillatory and ${\bf g}$ is ${\bf h}$-singular, then Plancherel's inversion formula and Cauchy-Schwartz inequality yield
$$
\lim_{n\rightarrow \infty}
\int_{\R}f_n(x)\overline{g_n(x)}dx =0,
$$
which implies that ${\bf f}$ and ${\bf g}$ are orthogonal in the sense that
$$
\|f_n+g_n\|^2_{L^2_x}=\|f_n\|^2_{L^2_x}+\|g_n\|^2_{L^2_x}+o(1) \,\,\, \textrm{ as }\,\,\, n\rightarrow \infty.
$$
Moreover, if ${\bf f}$ is a bounded sequence in $L^2(\R)$ and ${\bf h}$-oscillatory with respect to a scale ${\bf h}$, then ${\bf f}$ is ${\bf \widetilde{h}}$-singular for every scale ${\bf \widetilde{h}}$ orthogonal to ${\bf h}$.
\end{remark}

The following result, obtained by G\'erard \cite[Theorem 2.9]{Ge98}, is closely related to the above notations. It provides a decomposition of a bounded sequence in $L^2(\R)$ into a sum of oscillatory functions plus a singular error. This error goes to zero in an appropriate norm when the number of terms in the oscillatory sum goes to infinity.

\begin{proposition}\label{GeR}[G\'erard \cite[Theorem 2.9]{Ge98}]
 Let ${\bf f}=(f_n)_{n\in \mathbb{N}}$ be a bounded sequence in $L^2(\R)$. Then, there exists a subsequence, which we still denote by ${\bf f}=(f_n)_{n\in \mathbb{N}}$, a family $({\bf h}^j)_{j\in \mathbb{N}}$ of positive scales and a family $({\bf g}^j)_{j\in \mathbb{N}}$ of bounded functions in $L^2(\R)$ such that
 \begin{itemize}
 \item[(i)] for every $j \neq j'$, ${\bf h}^j$ and ${\bf h}^{j'}$ are orthogonal;
 \item[(ii)] for every $j\in \N$, ${\bf g}^j$ is ${\bf h}^j$-oscillatory;
 \item[(iii)] for every $J\geq1$ there exists a function $R^J_n \subset L^2(\R)$, which is ${\bf h}^j$-singular for $j=1,\dots,J$, satisfying
\begin{equation*}
f_n(x)=\sum_{j=1}^{J}g_n^j(x) + R_n^J(x), \,\, \textrm{ for all } \,\, x\in \R.
\end{equation*}
Moreover
\begin{equation}\label{SPWNL3}
\limsup_{n\rightarrow \infty} \|R_n^J\|_{\dot{B}^0_{2,\infty}} \rightarrow 0,  \peqq \textrm{as} \peqq J\rightarrow \infty,
\end{equation}
where $\|\cdot\|_{\dot{B}^0_{2,\infty}}$ is the Besov norm defined\footnote{For a rigorous definition of the Besov space $\dot{B}^0_{2,\infty}$ in terms of cut-off functions we refer the reader to the book Bergh and L{\"o}fstr{\"o}m \cite{BeLo}.} by
$$
\|f\|_{\dot{B}^0_{2,\infty}} = \sup_{k\in\Z}\left(\int_{2^k\leq |\xi|\leq 2^{k+1}}|\hat{f}(\xi)|^2 d\xi\right)^{1/2}.
$$
 \item[(iv)] for a fixed $J\geq 1$ we also have
\begin{equation}\label{PYTHA3}
\|f_n\|^2_{L^2_x}-\sum_{j=1}^{J}\|g_n^j\|^2_{L^2_x} - \|R_n^J\|^2_{L^2_x} =o(1),  \peqq \textrm{as} \peqq n\rightarrow \infty.
\end{equation}
 \end{itemize}
\end{proposition}

Next, we prove a profile decomposition result for ${\bf 1}$-oscillatory bounded sequence in $\dot{H}_x^{s_k}(\R)$. Later, the general case stated in Theorem \ref{profdec} will be reduced to this one using G\'erard's result in Proposition \ref{GeR}.

\begin{proposition}\label{1osc}
 Let $k>4$ and ${\bf u}=(u_n)_{n\in \mathbb{N}}$ be a bounded sequence in $\dot{H}_x^{s_k}(\R)$ and assume $D_x^{s_k}{\bf u}=(D_x^{s_k}u_n)_{n\in \mathbb{N}}$ is ${\bf 1}$-oscillatory. Then there exists a subsequence, which we still denote by ${\bf u}=(u_n)_{n\in \mathbb{N}}$, a family of parameters $(x_n^j, t_n^j)_{n\in \mathbb{N}, j\in \mathbb{N}}$ and sequence of functions $(\psi^j)_{j\in \mathbb{N}}\subset \dot{H}_x^{s_k}(\R)$ such that for every $J\geq1$ there exists $W^J_n\subset \dot{H}_x^{s_k}(\R)$ satisfying
\begin{equation*}
u_n(x)=\sum_{j=1}^{J}U\left(-t_n^j\right)\psi^j\left(x-x_n^j\right) + W_n^J(x).
\end{equation*}
Moreover, the sequence $D_x^{s_k}{{\bf {W}}^J}=(D_x^{s_k}W_n^J)_{n\in \mathbb{N}}$ is ${\bf 1}$-oscillatory for all $J \in \mathbb{N}$ and
\begin{equation}\label{SPWNL2}
\limsup_{n\rightarrow \infty} \|D_x^{1/p}U(t)W_n^J\|_{L^p_tL^{q}_x} \rightarrow 0,  \peqq \textrm{as} \peqq J\rightarrow \infty,
\end{equation}
for all $\dot{H}_x^{s_k}$-admissible pair $(p,q)$, with $k+2<p<\infty$, satisfying \eqref{Adm}.

Furthermore, for all $1\leq i \neq j\leq J$, we have
\begin{equation}\label{XT2}
\lim_{n\rightarrow \infty} \left|x_n^i-x_n^j\right|+\left|t_n^i-t_n^j\right|=\infty,
\end{equation}
and, for fixed $J\geq 1$
\begin{equation}\label{PYTHA2}
\|u_n\|^2_{\dot{H}_x^{s_k}}-\sum_{j=1}^{J}\|\psi^j\|^2_{\dot{H}_x^{s_k}} - \|W_n^J\|^2_{\dot{H}_x^{s_k}} =o(1),  \peqq \textrm{as} \peqq n\rightarrow \infty.
\end{equation}
\end{proposition}
\begin{proof}
First of all, by inequality \eqref{STR0}, we have
\begin{equation}\label{STR01}
 \|D^{1/(k+2)}_xU(t)u_0\|_{L_t^{k+2}L_x^{k(k+2)/4}}\leq c\|u_0\|_{\dot{H}_x^{s_k}},
\end{equation}
and
\begin{equation}\label{STR02}
 \|U(t)u_0\|_{L_t^{\infty}L_x^{k/2}}\leq c\|u_0\|_{\dot{H}_x^{s_k}},
\end{equation}
Therefore, if $\{W^J_n\}_{n, J \in \mathbb{N}}$ is a bounded sequence in $\dot{H}_x^{s_k}(\R)$, by complex interpolation between \eqref{STR01} and \eqref{STR02}, we have
\begin{eqnarray*}
 \|D_x^{1/p}U(t)W^J_n\|_{L^{p}_tL^{q}_x}&\leq&  \|D^{1/(k+2)}_xU(t)W^J_n\|_{L_t^{k+2}L_x^{k(k+2)/4}}^{\theta}  \|U(t)W^J_n\|_{L_t^{\infty}L_x^{k/2}}^{1-\theta}\\
 &\leq& \|W^J_n\|_{\dot{H}^{s_k}_x}^{\theta} \|U(t)W^J_n\|_{L^{\infty}_t L^{k/2}_x}^{1-\theta}\\
 &\leq& c \|U(t)W^J_n\|_{L^{\infty}_t L^{k/2}_x}^{1-\theta},
 \end{eqnarray*}
where $(p,q)$ is an $\dot{H}_x^{s_k}(\R)$-admissible pair satisfying \eqref{Adm} and $\theta=(k+2)/p \in (0,1)$, since $p>k+2$.
Hence, to obtain the desired result   it is sufficient  to show that
\begin{equation}\label{UWNL}
\limsup_{n\rightarrow \infty} \|U(t)W^J_n\|_{{L^{\infty}_t L^{k/2}_x}} \rightarrow 0,  \peqq \textrm{as} \peqq J\rightarrow \infty.
\end{equation}

To simplify the notation let us denote $S \equiv L^{\infty}_tL^{k/2}_x$. In the sequel, we follow the idea introduced by Fanelli and Visciglia \cite{FV13} (see also Bahouri and G{\'e}rard \cite{BaGe99} and Keraani \cite{Ke01}). We construct the desired sequences iteratively as follows. Let
$$
4A_1=\liminf_{n\rightarrow \infty}\|U(t)u_n\|_S.
$$
If $A_1=0$ we take, for all $j,J\in \N$, $\psi^j=0$, $t_n^j=x_n^j=0$ and $W^J_n=u_n'$, where $(u_n')_{n\in \N}$ is a subsequence of $(u_n)_{n\in \N}$ such that $\lim_{n\rightarrow \infty}\|U(t)u'_n\|_S=0$ and the proof is completed.
Suppose now $A_1>0$, without loss of generality we can assume $\|U(t)u_n\|_S\geq 2A_1$ for all $n\in \N$. By definition of $S$, there exists a sequence of times $(t_n^1)_{n\in \N}$ such that
\begin{equation}\label{beta1}
\|U(t^1_n)u_n\|_{L^{k/2}_x}\geq \frac{1}{2} \|U(t)u_n\|_S\geq A_1.
\end{equation}

Define the sequences of functions ${\boldsymbol{\phi}^1}=(\phi_n^1(x))_{n\in \N}$, where
\begin{equation}\label{beta2}
\phi_n^1(x)=U(t_n^1)u_n(x).
\end{equation}
Note that $(D_x^{s_k}{\phi}_n^1)_{n\in \N}$ is ${\bf 1}$-oscillatory since $(D_x^{s_k}u_n)_{n\in \N}$ is ${\bf 1}$-oscillatory. Let $\mathbb{P}(\boldsymbol{\phi}^1)$ denote the set of all  possible weak limits in $\dot{H}_x^{s_k}(\R)$ of all possible subsequences of $\boldsymbol{\phi}^1$ with all their possible translations. We define
$$
\gamma(\boldsymbol{\phi}^1)=\sup\{\|\psi\|_{\dot{H}_x^{s_k}}: \psi \in \mathbb{P}(\boldsymbol{\phi}^1)\}.
$$
Clearly $\gamma(\boldsymbol{\phi}^1) \leq \limsup_{n\rightarrow \infty}\|u_n\|_{\dot{H}_x^{s_k}}$. We also have $\gamma(\boldsymbol{\phi}^1)>0$. Indeed, since  $(D_x^{s_k}{\phi}_n^1)_{n\in \N}$ is ${\bf 1}$-oscillatory, this is a consequence of the following inequality proved by G\'erard \cite[estimate (4.19)]{Ge98}
\begin{equation}\label{Gereq}
\limsup_{n\rightarrow \infty}\|\phi_n^1\|_{L^{k/2}_x}\leq c \limsup_{n\rightarrow \infty}\|\phi_n^1\|_{\dot{H}_x^{s_k}}^{4/k}\gamma(\boldsymbol{\phi}^1)^{1-4/k}.
\end{equation}
Therefore if $\gamma(\boldsymbol{\phi}^1)=0$, then $\limsup_{n\rightarrow \infty}\|\phi_n^1\|_{L^{k/2}_x}=0$, which, in view of \eqref{beta1} and \eqref{beta2}, is a contradiction with $A_1>0$.
Moreover, by definition, there exists a sequence of translations $(x_n^1)_{n\in \N}$, such that up to a subsequence
\begin{equation}\label{weak1}
U(t_n^1)u_n(x+x_n^1)=\phi_n^1(x+x_n^1)\rightharpoonup \psi^1 \textrm{ weakly in } \dot{H}_x^{s_k}(\R),
\end{equation}
as $n\rightarrow \infty$, where
$$
0<\gamma(\boldsymbol{\phi}^1)\leq 2 \|\psi^1\|_{\dot{H}_x^{s_k}}.
$$
In particular $\psi^1 \neq 0$.

Now, let
\begin{equation}\label{Rn1}
W_n^1(x)=u_n(x)-U(-t_n^1)\psi^1(x-x^1_n)
\end{equation}
by \eqref{weak1}, it is clear that
$$
U(t_n^1)W_n^1(x+x_n^1)\rightharpoonup 0 \textrm{ weakly in } \dot{H}_x^{s_k}(\R),
$$
as $n\rightarrow \infty$.
Moreover, by the $\dot{H}^{s_k}_x(\R)$-preservation, it follows that
\begin{equation}\label{Rn1H}
\begin{split}
 \|W_n^1\|^2_{\dot{H}_x^{s_k}}&= \|u_n\|^2_{\dot{H}_x^{s_k}}+ \|U(-t_n^1)\psi^1(\cdot -x_n^1)\|^2_{\dot{H}_x^{s_k}}-2\Re(u_n,U(-t_n^1)\psi^1(\cdot -x_n^1))_{\dot{H}_x^{s_k}}\\
 &= \|u_n\|^2_{\dot{H}_x^{s_k}}+ \|\psi^1\|^2_{\dot{H}_x^{s_k}}-2\Re(\phi_n^1(\cdot +x_n^1),\psi^1)_{\dot{H}_x^{s_k}}\\
 &= \|u_n\|^2_{\dot{H}_x^{s_k}}- \|\psi^1\|^2_{\dot{H}_x^{s_k}}+o(1),
\end{split}
\end{equation}
as $n\rightarrow \infty$.

We also have that the sequence $(D_x^{s_k}W^1_n)_{n\in \N}$ is ${\bf 1}$-oscillatory. Indeed, let $\sigma_R(D)$ denote the multiplier operator with symbol $\sigma_R=1_{\{|\xi|\leq 1/R\}\cup\{|\xi|\geq R\}}$. Applying this operator to the both sides of \eqref{Rn1} and arguing as in \eqref{Rn1H} we obtain
$$
\|\sigma_R(D)W_n^1\|^2_{\dot{H}_x^{s_k}}= \|\sigma_R(D)u_n\|^2_{\dot{H}_x^{s_k}}- \|\sigma_R(D)\psi^1\|^2_{\dot{H}_x^{s_k}}+o(1), \textrm{ as } n\rightarrow \infty.
$$
Therefore
$$
\limsup_{n\rightarrow \infty}
\int\sigma_R(\xi)|\xi|^{2s_k}|\widehat{W^1_n}(\xi)|^2d \xi \leq \limsup_{n\rightarrow \infty} \int\sigma_R(\xi)|\xi|^{2s_k}|\widehat{u_n}(\xi)|^2d \xi,
$$
for every $R>0$.

Next, we define
$$
A_2= 4\liminf_{n\rightarrow \infty} \|U(t)W^1_n\|_{S}
$$
Here again, the only case we need to consider is $A_2>0$, otherwise we are done. Repeating the above procedure, with $u_n$ replaced by $W^1_n$ we can find a sequence of times $(t^2_n)_{n \in \mathbb{N}}$, a sequence of centers $(x^2_n)_{n \in \mathbb{N}}$ and a function $\psi^2\in \dot{H}_x^{s_k}(\R)$, such that the sequence ${\boldsymbol{\phi}^2}=(\phi_n^2)_{n\in N}$ given by $\phi_n^2(x)=U(t_n^2)W_n^1(x)$ satisfies
\begin{equation}\label{Rphi02}
\|\phi_n^2\|_{L^{k/2}_x}\geq \frac{1}{2} \|U(t)W^1_n\|_S\geq A_2.
\end{equation}
\begin{equation}\label{Rphi12}
U(t_n^2)W^1_n(x+ x_n^2) = \phi_n^2(x+x_n^2)\rightharpoonup \psi^2(x)\neq 0  \textrm{ weakly in } \dot{H}_x^{s_k}(\R),
\end{equation}
as $n\rightarrow \infty$,\\
and
\begin{equation}\label{Rphi22}
2\|\psi^2\|_{\dot{H}^{s_k}_x}\geq \gamma({\boldsymbol{\phi}^2})>0.
\end{equation}
Moreover, in view of the relation \eqref{Rn1H}, the sequence ${\boldsymbol{\phi}^2}$ is bounded in $\dot{H}^{s_k}_x(\R)$ and  $(D_x^{s_k}{\phi}_n^2)_{n\in \N}$ is ${\bf 1}$-oscillatory.

Let
\begin{equation}\label{Rn2}
W_n^2(x)=W_n^1(x)-U(-t_n^2)\psi^2(x- x_n^2)
\end{equation}
and note that \eqref{Rphi12}  yields
$$
U(t_n^2)W^2_n(x+ x_n^2) \rightharpoonup 0  \textrm{ weakly in } \dot{H}_x^{s_k}(\R),
$$
as $n\rightarrow \infty$. Similarly  to    \eqref{Rn1H} we obtain
\begin{equation}\label{Rn2H}
\|W_n^2\|^2_{\dot{H}_x^{s_k}} = \|W_n^1\|^2_{\dot{H}_x^{s_k}}- \|\psi^2\|^2_{\dot{H}_x^{s_k}} + o(1)
\end{equation}
as $n\rightarrow \infty$. Therefore,  combining \eqref{Rn1H} and \eqref{Rn2H}, we have
\begin{equation*}
\|W_n^2\|^2_{\dot{H}_x^{s_k}} = \|u_n\|^2_{\dot{H}_x^{s_k}}- \|\psi^1\|^2_{\dot{H}_x^{s_k}} -\|\psi^2\|^2_{\dot{H}_x^{s_k}} + o(1)
\end{equation*}
Moreover, relations \eqref{Rn1} and \eqref{Rn2} yield the decomposition
$$
u_n(x)=U(-t_n^1)\psi^1(x-x^1_n)+U(-t_n^2)\psi^2(x-x^2_n)+W_n^2(x),
$$
where the sequence $(D_x^{s_k}W_n^2)_{n\in \N}$ is ${\bf 1}$-oscillatory.

We claim that
\begin{equation*}
\lim_{n\rightarrow \infty} |x_n^1-x_n^2|+|t_n^1-t_n^2|=\infty.
\end{equation*}
If not, by compactness we can assume $x_n^1-x_n^2 \rightarrow \bar{x}\in \R$ and $t_n^1-t_n^2 \rightarrow \bar{t}\in \R$, as $n\rightarrow \infty$. By definition of $W_n^1$ and $W_n^2$, we have
\begin{equation*}
\begin{split}
& U(t_n^2-t_n^1)\left(U(t_n^1)u_n(x+x_n^1-(x_n^1-x_n^2))-\psi^1(x-(x_n^1-x_n^2))\right)\\
=& U(t_n^2)\left(u_n(x+x_n^2))-U(t_n^1)\psi^1(x-(x_n^1-x_n^2))\right)\\
=& U(t_n^2)W_n^1(x+x_n^2).
\end{split}
\end{equation*}
Using the relation \eqref{weak1} and the assumption that $\bar{x}$ and $\bar{t}$ are finite, the first line of the above expression converges weakly to $0$, as $n\rightarrow \infty$. On the other hand, by \eqref{Rphi12}, the last line converges weakly to $\psi^2(x)\neq 0$, as $n\rightarrow \infty$, which is a contradiction.

Next, we construct the functions $\psi^j\in \dot{H}_x^{s_k}(\R)$, $j>2$ by induction. Indeed, let $J>2$  and  assume there exists sequences of times $(t^j_n)_{n \in \mathbb{N}}$, sequences of centers $(x^j_n)_{n \in \mathbb{N}}$ satisfying \eqref{XT2} and functions $\psi^j, W_n^j\in \dot{H}_x^{s_k}(\R)$, such that for every $j\in \{1,\dots, J-1\}$ we have
\begin{equation}\label{Rnj}
W_n^{j}(x)=W_n^{j-1}(x)-U(-t_n^j)\psi^j(x- x_n^j),
\end{equation}
where $W_n^0=u_n$ and the sequence $(D_x^{s_k} W_n^j)_{n\in \N}$ is ${\bf 1}$-oscillatory.

For $\phi_n^j(x)=U(t_n^j)W^{j-1}_n(x)$
\begin{equation}\label{wnj}
\phi_n^j(x+x_n^j)\rightharpoonup \psi^j \neq 0 \textrm{ weakly in } \dot{H}_x^{s_k}(\R),
\end{equation}
as $n\rightarrow \infty$. Moreover, the sequences $(\phi_n^j)_{n\in \N}$ are bounded in $\dot{H}^{s_k}_x(\R)$, $(D_x^{s_k} \phi_n^j)_{n\in \N}$ are ${\bf 1}$-oscillatory and satisfy
\begin{equation}\label{wnj2}
\|\phi_n^j\|_{L^{k/2}_x}\geq \frac{1}{2} \|U(t)W^{j-1}_n\|_S\geq A_j,
\end{equation}
where
$$
A_j= 4\liminf_{n\rightarrow \infty} \|U(t)W^{j-1}_n\|_{S}
$$
In addition, the functions $u_n$ can be written as
$$
u_n(x)=U(-t_n^1)\psi^1(x-x_n^1)+\cdots+U(-t_n^{J-1})\psi^{J-1}(x-x_n^{J-1})+W_n^{J-1}(x),
$$
satisfying
\begin{equation}\label{Rnjsk}
\|W_n^{J-1}\|^2_{\dot{H}_x^{s_k}} = \|u_n\|^2_{\dot{H}_x^{s_k}}- \|\psi^1\|^2_{\dot{H}_x^{s_k}} -\cdots-\|\psi^{J-1}\|^2_{\dot{H}_x^{s_k}} + o(1)
\end{equation}
as $n\rightarrow \infty$.

Now, we obtain $(t^J_n)_{n \in \mathbb{N}}$, sequences of centers $(x^J_n)_{n \in \mathbb{N}}$ satisfying \eqref{XT2} and functions $\psi^J, W_n^J\in \dot{H}_x^{s_k}(\R)$, satisfying the desired properties.  Define
$$
A_J= \liminf_{n\rightarrow \infty} \|U(t)W^{J-1}_n\|_{S}.
$$
If $A_J=0$ the proof is completed by taking $\psi^j=0$, for all $j\geq J$. Assume $A_J>0$ and apply the previous  procedure to the sequence $(W_n^{J-1})_{n\in \mathbb{N}}$ to obtain, passing to a subsequence if necessary, sequences of real numbers $(t^J_n)_{n \in \mathbb{N}}$, $(x^J_n)_{n \in \mathbb{N}}$ and a function $\psi^J\in \dot{H}_x^{s_k}(\R)$ such that
\begin{equation}\label{Rphij}
U(t_n^J)W^{J-1}_n(x+ x_n^J) = \phi_n^J(x+x_n^J)\rightharpoonup \psi^J(x)\neq 0  \textrm{ weakly in } \dot{H}_x^{s_k}(\R),
\end{equation}
as $n\rightarrow \infty$.

Next, we prove conditions \eqref{XT2} and \eqref{PYTHA2} for $j, k \in \{1,\dots,J\}$. First, by induction hypothesis  \eqref{XT2} holds for $j, k \in \{1,\dots,J-1\}$ and hence by definition we have
\begin{eqnarray*}
W_n^{J-1}&=&W_n^{j-1} - U(-t_n^j)\psi^j(\cdot-x_n^j)-U(-t_n^{j+1})\psi^{j+1}(\cdot-x_n^{j+1})\\
&&-\cdots - U(-t_n^{J-1})\psi^{J-1}(\cdot-x_n^{J-1}).
\end{eqnarray*}

Therefore, applying the operator $U(t_n^j)$ and the shift $x_n^j$, we obtain
\begin{equation*}
\begin{split}
U(t_n^j)W_n^{J-1}(\cdot+x_n^j)=& U(t_n^j)W_n^{j-1}(\cdot+x_n^j) - \psi^j \\
&-\sum_{k=j+1}^{J-1} U(t_n^j-t_n^k)\psi^k(\cdot+x_n^j-x_n^k).
\end{split}
\end{equation*}
The difference of the first two  terms on the right-hand side of the last equation goes to zero weakly in $\dot{H}_x^{s_k}(\R)$ by \eqref{wnj}, and the same happens with the summation term by our induction hypotheses. Hence,
\begin{equation}\label{URj1}
U(t_n^j)W_n^{J-1}(x+x_n^j) \rightharpoonup 0\textrm{ weakly in } \dot{H}_x^{s_k}(\R),
\end{equation}
as $n\rightarrow \infty$.
On the other hand, we have the following identity
$$
U(t_n^J)W_n^{J-1}(x+x_n^J)=U(t_n^j+(t_n^J-t_n^j))W_n^{J-1}(x+x_n^j+(x_n^J-x_n^j))
$$
If $(x_n^J,t_n^J)$ and $(x_n^j,t_n^j)$ are not orthogonal, we can assume $x_n^J-x_n^j \rightarrow \bar{x}\in \R$ and $t_n^J-t_n^j \rightarrow \bar{t}\in \R$. By combining this fact with \eqref{URj1} we deduce that
$$
U(t_n^J)W_n^{J-1}(x+x_n^J) \rightharpoonup 0\textrm{ weakly in } \dot{H}_x^{s_k}(\R),
$$
as $n\rightarrow \infty$, which is a contradiction with \eqref{Rphij}.

Next, we turn to the proof of \eqref{PYTHA2}. Recall that
$$
W_n^J(x)=W^{J-1}_n(x)-U(-t_n^J)\psi^J(x -x_n^J).
$$
By \eqref{Rnjsk}, the asymptotic Pythagorean expansion \eqref{PYTHA2} holds at rank $J-1$. Using the weak convergence \eqref{Rphij} and expanding
$$
 \|W_n^J\|^2_{\dot{H}_x^{s_k}}= \|U(t_n^J)W^J_n(\cdot +x_n^J)\|^2_{\dot{H}_x^{s_k}}=\|U(t_n^J)W^{J-1}_n(\cdot +x_n^J)-\psi^J\|^2_{\dot{H}_x^{s_k}}
$$
as an inner product, it is easy to conclude \eqref{PYTHA2} at rank $J$ using \eqref{URj1}.

Finally, we consider the smallness property \eqref{SPWNL2}. Recall from  the beginning of the proof that it is sufficient to prove \eqref{UWNL}.  First, we observe that
$$
\sum_{j\geq 1}\|\psi^j\|^2_{\dot{H}_x^{s_k}}\leq \limsup_{n\rightarrow \infty}\|u_n\|^2_{\dot{H}_x^{s_k}}
$$
by \eqref{PYTHA2}. Now, fix $\varepsilon>0$, using the fact that $(u_n)_{n\in \N}$ is bounded in $\dot{H}_x^{s_k}(\R)$, there exists $J=J(\varepsilon)$ such that for any $j\geq J(\varepsilon)$
$$
\gamma(\boldsymbol{\phi}^j)\leq 2\|\psi^j\|_{\dot{H}_x^{s_k}}<2\varepsilon.
$$

Since $(D_x^{s_k}\phi_n^j)_{n\in \N}$ is ${\bf 1}$-oscillatory and uniformly bounded in ${L}_x^{2}(\R)$ for every $j\in \N$, we can apply inequality \eqref{Gereq} to deduce that for every $j\geq J(\varepsilon)$
$$
\limsup_{n\rightarrow \infty}\|\phi^j_n\|_{L_x^{k/2}}\leq c\varepsilon,
$$
for some constant $c>0$ independent of $\varepsilon$ and $j\geq J(\varepsilon)$. Thus, using inequality \eqref{wnj2} we conclude \eqref{SPWNL2}.
\end{proof}

\begin{remark}\label{Remvn}
Let $v_n(x,t)=U(t)u_n(x)$. The last  proposition gives us the following decomposition
\begin{equation*}
v_n(x,t)=\sum_{j=1}^{J}U\left(t-t_n^j\right)\psi^j\left(x-x_n^j\right) + U(t)W_n^J(x),
\end{equation*}
satisfying the relations \eqref{SPWNL2}, \eqref{XT2} and \eqref{PYTHA2}.
\end{remark}

Now, we are in position to prove the general profile decomposition in $\dot{H}_x^{s_k}(\R)$.

\begin{proof}[Proof of Theorem \ref{profdec}]
We follow closely the proof of Theorem 1.6 from Keraani's work \cite{Ke01}.  By assumption $(D^{s_k}_xu_n)_{n\in \mathbb{N}}$ is a bounded sequence in $L_x^2(\R)$. Applying Proposition \ref{GeR} to this sequence, we obtain a family of positive orthogonal scales $({\bf h}^j)_{j\in \mathbb{N}}$ satisfying \eqref{ho3} and a family $({\bf g}^j)_{j\in \mathbb{N}}$ of bounded functions in $\dot{H}_x^{s_k}(\R)$, such that
\begin{equation}\label{DecUn}
u_n(x)=\sum_{j=1}^{J}g_n^j(x) + R_n^J(x), \,\, \textrm{ for all } \,\, x\in \R,
\end{equation}
where $(D^{s_k}_xg^j_n)_{n\in \mathbb{N}}$ is ${{\bf h}^j}$-oscillatory for every $j\in \N$ and $(D^{s_k}_xR^j_n)_{n\in \mathbb{N}}$ is ${{\bf h}^j}$-singular for every $j\in \{1,\dots, J\}$ and
\begin{equation}\label{DskR1}
\limsup_{n\rightarrow \infty} \|D^{s_k}_xR_n^J\|_{\dot{B}^0_{2,\infty}} \rightarrow 0,  \peqq \textrm{as} \peqq J\rightarrow \infty.
\end{equation}
Furthermore, the following asymptotic Pythagorean expansion holds for every $J\geq 1$
\begin{equation}\label{PYTHA4}
\|u_n\|^2_{\dot{H}_x^{s_k}}-\sum_{j=1}^{J}\|g_n^j\|^2_{\dot{H}_x^{s_k}} - \|R_n^J\|^2_{\dot{H}_x^{s_k}} =o(1),  \peqq \textrm{as} \peqq n\rightarrow \infty.
\end{equation}
Applying the propagator $\{U(t)\}_{t\in \R}$ in both sides of equality \eqref{DecUn} we obtain a decomposition of $v_n(x,t)=U(t)u_n(x)$ satisfying
\begin{equation}\label{vn}
v_n(x,t)=\sum_{j=1}^{J}U(t)g_n^j(x) + U(t)R_n^J(x), \,\, \textrm{ for all } \,\, (x,t)\in \R\times \R,
\end{equation}
Moreover, since $\{U(t)\}_{t\in \R}$ is unitary in $\dot{H}_x^{s_k}(\R)$, in view of \eqref{PYTHA4}, we also have for every $J\geq 1$ and $t\in \R$
\begin{equation}\label{PYTHA5}
\|v_n(\cdot, t)\|^2_{\dot{H}_x^{s_k}}-\sum_{j=1}^{J}\|g_n^j\|^2_{\dot{H}_x^{s_k}} - \|R_n^J\|^2_{\dot{H}_x^{s_k}} =o(1),  \peqq \textrm{as} \peqq n\rightarrow \infty.
\end{equation}

Next, we prove the smallness of the remainder $U(t)R_n^J(x)$ in the Strichartz norm $L_t^{\infty}L_x^{k/2}$. This is based on the following refined Sobolev inequality proved by Shao \cite{Sh091} (see also G\'erard, Meyer and Oru \cite{GMO96} and Keraani \cite{Ke01}).
\begin{lemma}\label{Shao09}[Shao\cite[inequality (4.1)]{Sh091}]
For any $1<r<\infty$, and  $s\geq 0$ satisfying  $\dfrac{1}{r}+\dfrac{s}{d}=\dfrac{1}{2}$ we have
\begin{equation*}
\|f\|_{L^{r}_x}\leq c \|D^{s}_xf\|^{1-2s/d}_{L^{2}_x}\|D^{s}_xf\|^{2s/d}_{\dot{B}^0_{2,\infty}} .
\end{equation*}
\end{lemma}
In particular, for space dimension one and $s_k=\dfrac{1}{2}-\dfrac{2}{k}$, we have
\begin{equation}\label{RefSob}
\|f\|_{L^{k/2}_x}\leq c \|D^{s_k}_xf\|^{4/k}_{L^{2}_x}\|D^{s_k}_xf\|^{1-4/k}_{\dot{B}^0_{2,\infty}} .
\end{equation}
Next, since $U(t)R_n^J(x)$ is a solution of the linear equation \eqref{Lkdv} it is clear that $\sigma_k(D)U(t)R_n^J(x)$ is also a solution of the same equation, where $\sigma_k=1_{\{2^k\leq |\xi|\leq 2^{k+1}\}}$. Using again that $\{U(t)\}_{t\in \R}$ is a unitary group in $\dot{H}_x^{s_k}(\R)$, we obtain
\begin{equation}\label{DskR2}
\|D^{s_k}_xU(t)R_n^J\|_{\dot{B}^0_{2,\infty}}=\|D^{s_k}_xR_n^J\|_{\dot{B}^0_{2,\infty}}.
\end{equation}
Putting together \eqref{DskR1} and \eqref{DskR2}, we conclude that
\begin{equation}\label{DskR3}
\limsup_{n\rightarrow \infty} \|D^{s_k}_xU(t)R_n^J\|_{L^{\infty}_t\dot{B}^0_{2,\infty}} \rightarrow 0,  \peqq \textrm{as} \peqq J\rightarrow \infty.
\end{equation}

Now, applying \eqref{RefSob} to $U(t)R_n^J(x)$ and using \eqref{PYTHA4} we have
\begin{equation*}
\begin{split}
\limsup_{n\rightarrow \infty}\|U(t)R_n^J\|_{L^{\infty}_tL^{k/2}_x}&\leq c \limsup_{n\rightarrow \infty}\|U(t)R_n^J\|^{4/k}_{L^{\infty}_t\dot{H}_x^{s_k}}\limsup_{n\rightarrow \infty}\|D^{s_k}_xU(t)R_n^J\|^{1-4/k}_{L^{\infty}_t\dot{B}^0_{2,\infty}} \\
&\leq c\limsup_{n\rightarrow \infty}\|u_n\|^{4/k}_{\dot{H}_x^{s_k}}\limsup_{n\rightarrow \infty}\|D^{s_k}_xU(t)R_n^J\|^{1-4/k}_{L^{\infty}_t\dot{B}^0_{2,\infty}}
\end{split}
\end{equation*}
Therefore the limit \eqref{DskR3} yields
\begin{equation*}
\limsup_{n\rightarrow \infty}\|U(t)R_n^J\|_{L^{\infty}_tL^{k/2}_x} \rightarrow 0,  \peqq \textrm{as} \peqq J\rightarrow \infty.
\end{equation*}
Moreover, arguing as in the beginning of the proof of Proposition \ref{1osc}, we also have
\begin{equation}\label{SPWNL5}
\limsup_{n\rightarrow \infty}\|D_x^{1/p}U(t)R_n^J\|_{L^{p}_tL^{q}_x} \rightarrow 0,  \peqq \textrm{as} \peqq J\rightarrow \infty,
\end{equation}
for all $\dot{H}_x^{s_k}$-admissible pair $(p,q)$, with $k+2<p<\infty$, satisfying \eqref{Adm}.

Let
\begin{equation}\label{Psi}
\Psi_n^j(y,s)=(h_n^j)^{2/k}\left(U((h_n^j)^3s)g_n^j\right)(h_n^jy).
\end{equation}
Since the sequence $(D^{s_k}_xg^j_n)_{n\in \mathbb{N}}$ is bounded in $L_x^2(\R)$ and ${{\bf h}^j}$-oscillatory the sequence $(D^{s_k}_x\Psi_n^j(y,0))_{n\in \mathbb{N}}$ is bounded in $L_x^2(\R)$ and ${\bf 1}$-oscillatory. Indeed, for $j\in \N$
$$
\|D^{s_k}_x\Psi_n^j(y,\cdot)\|_{L^{2}_x}=\|D^{s_k}_xg_n^j\|_{L^{2}_x}.
$$
Moreover
$$
\int_{|\xi|\leq \frac{1}{R}}|(D^{s_k}_x\Psi_n^j)^{\widehat{}}(\xi,0)|^2d \xi = \int_{h_n^j|\xi|\leq \frac{1}{R}}|(D^{s_k}_xg_n^j)^{\widehat{}}(\xi)|^2d \xi
$$
and
$$
\int_{|\xi|\geq {R}}|(D^{s_k}_x\Psi_n^j)^{\widehat{}}(\xi,0)|^2d \xi = \int_{h_n^j|\xi|\geq {R}}|(D^{s_k}_xg_n^j)^{\widehat{}}(\xi)|^2d \xi
$$
Also, note that $\Psi_n^j(y,s)=U(s)f_n^j(y)$, where $f_n^j(y)=(h_n^j)^{2/k}g_n^j(h_n^jy)$. Indeed
\begin{equation*}
\begin{split}
(\Psi_n^j)^{\widehat{}}(\xi,s)&=\frac{(h_n^j)^{2/k}}{h_n^j}\left(U((h_n^j)^3s)g_n^j\right)^{\widehat{}}\left(\frac{\xi}{h_n^j}\right)\\
&=\frac{(h_n^j)^{2/k}}{h_n^j}e^{is\xi^3}\widehat{g}_n^j\left(\frac{\xi}{h_n^j}\right)=\left(U(s)f_n^j\right)^{\widehat{}}(y).
\end{split}
\end{equation*}

For every $j\geq 1$ we apply Remark \ref{Remvn} to the sequence $(\Psi_n^j)_{n\in \mathbb{N}}$. Therefore, using a diagonal extraction and passing to a subsequence if necessary, we obtain a family $(\psi^{(j,\alpha)})_{\alpha\in \N}$ and a family $(y_n^{(j,\alpha)}, s_n^{(j,\alpha)})_{\alpha\in \N}$ such that
\begin{equation*}
\Psi_n^j(x,t)=\sum_{\alpha=1}^{A_j}U\left(t-s_n^{(j,\alpha)}\right)\psi^{(j,\alpha)}\left(x-y_n^{(j,\alpha)}\right) + U(t)W^{(j,A_j)}_n(x).
\end{equation*}
From \eqref{Psi} we have
$$
U(t)g_n^j(x)=\frac{1}{(h_n^j)^{2/k}}\Psi_n^j\left(\frac{x}{h_n^j}, \frac{t}{(h_n^j)^3}\right),
$$
and hence, we deduce that
\begin{equation}\label{Utgnj}
U(t)g_n^j(x)=\sum_{\alpha=1}^{A_j}\frac{1}{(h_n^j)^{2/k}}U\left(\frac{t-t_n^{(j,\alpha)}}{(h_n^j)^3}\right)\psi^{(j,\alpha)}\left(\frac{x-x_n^{(j,\alpha)}}{h_n^j}\right) + w^{(j,A_j)}_n(x,t),
\end{equation}
where
\begin{equation}\label{ynsn}
x_n^{(j,\alpha)}=h_n^jy_n^{(j,\alpha)},\,\,\,\, t_n^{(j,\alpha)}=(h_n^j)^3s_n^{(j,\alpha)}
\end{equation}
and
$$
w^{(j,A_j)}_n(x,t)=\frac{1}{(h_n^j)^{2/k}}U\left(\frac{t}{(h_n^j)^3}\right)W^{(j,A_j)}_n\left(\frac{x}{h_n^j}\right).
$$
We recall that $(D_x^{s_k}W^{(j,A_j)}_n)_{n\in\N}$ is ${\bf 1}$-oscillatory, and therefore $(D_x^{s_k} w^{(j,A_j)}_n)_{n\in\N}$ is ${{\bf h}^j}$-oscillatory.
Hence,  the last equation,  \eqref{Utgnj} and  \eqref{vn} yield
\begin{eqnarray*}
v_n(x,t)&=&\sum_{j=1}^{J}\left(\sum_{\alpha=1}^{A_j}\frac{1}{(h_n^j)^{2/k}}U\left(\frac{t-t_n^{(j,\alpha)}}{(h_n^j)^3}\right)\psi^{(j,\alpha)}\left(\frac{x-x_n^{(j,\alpha)}}{h_n^j}\right)\right) \\
&&+ w_n^{(J,A_1,\dots,A_J)}(x,t),
\end{eqnarray*}
where
$$
w_n^{(J,A_1,\dots,A_J)}(x,t)= \sum_{j=1}^Jw_n^{(j,A_j)}(x,t)+U(t)R_n^J(x)
$$

Next, we enumerate the pairs $(j,\alpha)$ by $m$ satisfying
\begin{equation}\label{m}
m(j,\alpha) < m(l,\beta)\,\, \textrm{ if } \,\, j+\alpha<l+\beta \,\, \textrm{ or } \,\, j+\alpha=l+\beta \,\, \textrm{ and }\,\, j<l.
\end{equation}

We first consider the asymptotic Pythagorean expansion \eqref{PYTHA}. Combining \eqref{PYTHA5} and Remark \ref{Remvn} we conclude that
\begin{equation}\label{PYTHA6}
\|v_n(\cdot, t)\|^2_{\dot{H}_x^{s_k}}-\sum_{j=1}^{J}\sum_{\alpha=1}^{A_j}\|\psi^{(j,\alpha)}\|^2_{\dot{H}_x^{s_k}} - \|w_n^{(J,A_1,\dots,A_J)}\|^2_{\dot{H}_x^{s_k}} =o(1),  \peqq \textrm{as} \peqq n\rightarrow \infty.
\end{equation}
where we have used Remark \ref{Orth}, the fact that  $(D_x^{s_k} w_n^{j,A_j})_{n\in \N }$ is ${{\bf h}^j}$-oscillatory, $(D_x^{s_k} R_n^J)_{n\in \N }$ is ${{\bf h}^j}$-singular for all $j\in \{1,\dots, J\}$ and ${{\bf h}^j}$, ${{\bf h}^{j'}}$ are orthogonal for every $j\neq j'$, to obtain
$$
\|w_n^{(J,A_1,\dots,A_J)}\|^2_{\dot{H}_x^{s_k}} = \sum_{j=1}^{J}\|w^{(j,A_j)}_n\|^2_{\dot{H}_x^{s_k}}+\|R_n^J\|^2_{\dot{H}_x^{s_k}}.
$$

Next, we prove that the family $((h_n^m)_{n\in \N}, (x_n^m)_{n\in \N},  (t_n^m)_{n\in \N})_{m\in \N}$ ($m$ given by \eqref{m}) is pairwise orthogonal in the sense that relation \eqref{XT} holds. We have two possibilities.

First, if
$$
(h_n^m, x_n^m, t_n^m)=(h_n^i, x_n^{(i,\alpha)}, t_n^{(i,\alpha)}) \textrm{ and } (h_n^j, x_n^j, t_n^j)=(h_n^l, x_n^{(l,\beta)}, t_n^{(l,\beta)}), \textrm{ with } i\neq l
$$
then by Proposition \ref{GeR} we have
\begin{equation*}
\lim_{n\rightarrow \infty}
\frac{h^i_n}{h^l_n} + \frac{h^l_n}{h^i_n} =+\infty.
\end{equation*}

Second, if
$$
(h_n^m, x_n^m, t_n^m)=(h_n^l, x_n^{(l,\alpha)}, t_n^{(l,\alpha)}) \textrm{ and } (h_n^j, x_n^j, t_n^j)=(h_n^l, x_n^{(l,\beta)}, t_n^{(l,\beta)}), \textrm{ with } \alpha\neq \beta
$$
then
\begin{equation}\label{ynsn2}
\left|\dfrac{x_n^{(l,\alpha)}-x_n^{(l,\beta)}}{h_n^l}\right|+\left|\dfrac{t_n^{(l,\alpha)}-t_n^{(l,\beta)}}{(h_n^l)^{3}}\right|=
\left|y_n^{(l,\alpha)}-y_n^{(l,\beta)}\right|+\left|s_n^{(l,\alpha)}-s_n^{(l,\beta)}\right|
\end{equation}
where we have used the relations \eqref{ynsn}. Therefore, the right hand side of \eqref{ynsn2} goes to infinity, as $n\rightarrow \infty$, by \eqref{XT2}.

Finally, we consider the smallness of the remainder in the Strichartz norm \eqref{SPWNL}. We start with the $\dot{H}_x^{s_k}$-admissible pair $p=q=3k/2$. By our definition of the enumeration $m$ \eqref{m}, it suffices to show
\begin{equation}\label{SPWNL4}
\limsup_{n\rightarrow \infty} \|D_x^{2/3k}w_n^{(J,A_1,\dots,A_J)}\|_{L^{3k/2}_{t,x}} \rightarrow 0,  \peqq \textrm{as} \peqq \inf_{1\leq j\leq J}\{J, j+A_j\}\rightarrow \infty,
\end{equation}

Let $\varepsilon>0$. By the limit \eqref{SPWNL5} with $p=q=3k/2$, there exists $J_0\geq 1$ such that for all $J\geq J_0$
\begin{equation}\label{SPWNL6}
\limsup_{n\rightarrow \infty} \|D_x^{2/3k}U(t)R_n^J\|_{L^{3k/2}_{t,x}} \leq \varepsilon/3.
\end{equation}
Moreover, by Remark \ref{Remvn}, for all $J\geq J_0$, there exists $B_J$ such that for all $A_j\geq B_J$ and each $j\in \{1,\dots, J\}$
\begin{equation}\label{SPWNL7}
\limsup_{n\rightarrow \infty} \|D_x^{2/3k}w_n^{(j,A_j)}\|_{L^{3k/2}_{t,x}} \leq \varepsilon/3J.
\end{equation}

Now, we decompose the remainder $w_n^{(J,A_1,\dots,A_J)}$ in the following way
\begin{eqnarray*}
w_n^{(J,A_1,\dots,A_J)}(x,t)&=& U(t_n)R_n^{J}(x) +\sum_{1\leq j\leq J, A_j<B_J} w_n^{(j,A_j)}(x,t)+ \\
&&+\sum_{1\leq j\leq J, A_j\geq B_J} w_n^{(j,A_j)}(x,t)\\
&=&U(t_n)R_n^{J}(x) +\sum_{1\leq j\leq J, A_j<B_J} w_n^{(j,B_J)}(x,t)\\
&&+ \sum_{1\leq j\leq J, A_j\geq B_J} w_n^{(j,A_j)}(x,t) +r_n^{(J,A_1,\dots,A_J)}(x,t),
\end{eqnarray*}
where
\begin{equation*}
r_n^{(J,A_1,\dots,A_J)}(x,t)=\sum_{1\leq j\leq J, A_j<B_J} (w_n^{(j,A_j)}(x,t)-w_n^{(j,B_J)}(x,t))
\end{equation*}
Using the expression \eqref{Utgnj}, we can rewrite the last identity as follows
\begin{equation*}
r_n^{(J,A_1,\dots,A_J)}(x,t)=\!\!\!\!\!\sum_{1\leq j\leq J, A_j<B_J}\sum_{\alpha=A_j+1}^{B_J} \frac{1}{(h_n^j)^{2/k}}U\left(\frac{t-t_n^{(j,\alpha)}}{(h_n^j)^3}\right)\psi^{(j,\alpha)}\left(\frac{x-x_n^{(j,\alpha)}}{h_n^j}\right)\!\!.
\end{equation*}
From inequalities \eqref{SPWNL6} and \eqref{SPWNL7} we infer that
\begin{equation}\label{SPWNL8}
\limsup_{n\rightarrow \infty} \|D_x^{2/3k}w_n^{(J,A_1,\dots,A_J)}\|_{L^{3k/2}_{t,x}} \leq 2\varepsilon/3+\limsup_{n\rightarrow \infty} \|D_x^{2/3k}r_n^{(J,A_1,\dots,A_J)}\|_{L^{3k/2}_{t,x}}
\end{equation}
We claim that
\begin{equation}\label{SPWNL9}
\limsup_{n\rightarrow \infty} \|D_x^{2/3k}r_n^{(J,A_1,\dots,A_J)}\|^{3k/2}_{L^{3k/2}_{t,x}}=\sum_{1\leq j\leq J, A_j<B_J}\sum_{\alpha=A_j+1}^{B_J}\|D_x^{2/3k}U(t)\psi^{(j,\alpha)}(x)\|_{L^{3k/2}_{t,x}}^{3k/2}
\end{equation}
Assuming the claim for a moment, we can use the Strichartz estimate \eqref{STR} to obtain
\begin{equation}\label{SPWNL10}
\limsup_{n\rightarrow \infty} \|D_x^{2/3k}r_n^{(J,A_1,\dots,A_J)}\|_{L^{3k/2}_{t,x}}\leq c \left(\sum_{1\leq j\leq J, A_j<\alpha}\|\psi^{(j,\alpha)}\|^{3k/2}_{\dot{H}_x^{s_k}}\right)^{2/3k},
\end{equation}
where the constant $c>0$ is independent of $j, \alpha, A_j$.
By relation \eqref{PYTHA6}, the series $\sum_{j,\alpha}\|\psi^{(j,\alpha)}\|^{2}_{\dot{H}_x^{s_k}}$ is convergent and since $3k/4>1$ for $k\geq 4$, the series $\sum_{j,\alpha}\|\psi^{(j,\alpha)}\|^{3k/2}_{\dot{H}_x^{s_k}}$ also converges. Thus, the right hand side of \eqref{SPWNL10} can be taken sufficiently small if $\inf_{1\leq j\leq J}\{J, j+A_j\}$ is large enough. Combining this information with inequality \eqref{SPWNL8} we conclude
\begin{equation*}
\limsup_{n\rightarrow \infty} \|D_x^{2/3k}w_n^{(J,A_1,\dots,A_J)}\|_{L^{3k/2}_{t,x}} \leq \varepsilon,
\end{equation*}
for $\inf_{1\leq j\leq J}\{J, j+A_j\}$ large enough, which implies the limit \eqref{SPWNL4}.

Now we consider the smallness of the remainder in the Strichartz norm \eqref{SPWNL}, for every $\dot{H}_x^{s_k}$-admissible pair $(p,q)$ satisfying \eqref{Adm}, with $k+2<p<\infty$. Indeed, arguing as in the beginning of the proof of Proposition \ref{1osc} interpolating  \eqref{STR} and \eqref{STR2} if $k+2<p<3k/2$ or interpolating \eqref{STR} and \eqref{STR1} if $3k/2<p<\infty$ we obtain the desired conclusion \eqref{SPWNL}.

To complete the proof, it remains to prove the claim \eqref{SPWNL9}. Since the family $((h_n^m)_{n\in \N}, (x_n^m)_{n\in \N},  (t_n^m)_{n\in \N})_{m\in \N}$ are pairwise orthogonal, it is sufficient to prove the following equality
\begin{equation*}
\limsup_{n\rightarrow \infty} \left\|D_x^{2/3k}\sum_{j=1}^{J}\dfrac{1}{(h_n^j)^{2/k}}U\left(\dfrac{t-t_n^j}{(h_n^j)^{3}}\right)\psi^j\left(\dfrac{x-x_n^j}{h_n^j}\right)\right\|^{3k/2}_{L^{3k/2}_{t,x}}
\end{equation*}
\begin{equation*}
=\sum_{ j=1}^{J}\|D_x^{2/3k}U(t)\psi^{j}(x)\|_{L^{3k/2}_{t,x}}^{3k/2},
\end{equation*}
if $((h_n^j)_{n\in \N}, (x_n^j)_{n\in \N},  (t_n^j)_{n\in \N})_{j\in \N}$ are pairwise orthogonal.

To simplify the notation, let
\begin{equation}\label{Not1}
V^j_n(x,t)=\dfrac{1}{(h_n^j)^{2/k}}U\left(\dfrac{t-t_n^j}{(h_n^j)^{3}}\right)\psi^j\left(\dfrac{x-x_n^j}{h_n^j}\right)
\end{equation}
and
\begin{equation}\label{Not2}
V^j(x,t)=U(t)\psi^j(x).
\end{equation}
It is clear that
$$
V^j_n(x,t)=\dfrac{1}{(h_n^j)^{2/k}}V^j\left(\dfrac{x-x_n^j}{h_n^j}, \dfrac{t-t_n^j}{(h_n^j)^{3}}\right).
$$
Moreover, a simple computation shows
\begin{equation}\label{chain_ruleDx}
D_x^{2/3k}V^j_n(x,t)=\dfrac{1}{(h_n^j)^{2/k+2/3k}}(D_x^{2/3k}V^j)\left(\dfrac{x-x_n^j}{h_n^j}, \dfrac{t-t_n^j}{(h_n^j)^{3}}\right)
\end{equation}
and
\begin{equation}\label{Vjn}
\left\|D_x^{2/3k}V^j_n\right\|_{L^{3k/2}_{t,x}}=\left\|D_x^{2/3k}V^j\right\|_{L^{3k/2}_{t,x}}.
\end{equation}
We will need the following elementary inequality (see G\'erard \cite{Ge98})
\begin{equation}\label{GeL}
\left| | \sum_{j=1}^{J} a_j |^p - \sum_{j=1}^{J}\left| a_j\right|^{p}\right|\leq c_J\sum_{j\neq l} |a_j||a_l|^{p-1},
\end{equation}
for all $p\in [2,\infty)$. Hence, combining \eqref{Vjn} and \eqref{GeL}, we obtain
\begin{equation*}
\left| \left\|D_x^{2/3k}\sum_{j=1}^{J} V^j_n\right\|^{3k/2}_{L^{3k/2}_{t,x}} - \sum_{j=1}^{J}\left\|D_x^{2/3k}V^j\right\|^{3k/2}_{L^{3k/2}_{t,x}}\right|
\end{equation*}
\begin{equation*}
\leq c_J\sum_{j\neq l}\int \int  |D_x^{2/3k}V^j_n||D_x^{2/3k}V^l_n|^{3k/2-1}dxdt.
\end{equation*}

Next, we prove that the right hand side of the last inequality goes to zero as $n\rightarrow \infty$. Indeed, since $((h_n^j)_{n\in \N}, (x_n^j)_{n\in \N},  (t_n^j)_{n\in \N})_{j\in \N}$ are pairwise orthogonal, we have either
\begin{equation}\label{hnxntn}
\lim_{n\rightarrow \infty}
\frac{h^j_n}{h^l_n} + \frac{h^l_n}{h^j_n} =+\infty
\end{equation}
or
\begin{equation}\label{hnxntn2}
h^j_n=h^l_n \,\, \textrm{ and } \,\, \lim_{n\rightarrow \infty} \left|\dfrac{x_n^{j}-x_n^{l}}{h_n^j}\right|+\left|\dfrac{t_n^{j}-t_n^{l}}{(h_n^j)^{3}}\right|= \infty.
\end{equation}

First,  suppose that $D_x^{2/3k}V^j$,  $j\in\{1,\dots,J\}$ are continuous and compactly supported.  If \eqref{hnxntn} holds,  without loss of generality, we assume
\begin{equation}\label{hnhl}
\frac{h^j_n}{h^l_n} \rightarrow \infty,
\end{equation}
as $n\rightarrow \infty$ (the other case is analogous). Equation (\ref{chain_ruleDx}) and the change of variables $x=h_n^ly+x_n^l$ and $t=(h_n^l)^3s+t_n^l$ yields
\begin{equation*}
\int \int  |D_x^{2/3k}V^j_n(x,t)||D_x^{2/3k}V^l_n(x,t)|^{3k/2-1}dxdt
\end{equation*}
\begin{equation*}
\begin{split}
&= \left(\frac{h_n^l}{h_n^j}\right)^{8/3k} \int \int  |D_x^{2/3k}V^j\left(y', s'\right)||D_x^{2/3k}V^l(y,s)|^{3k/2-1}dyds\\
&\leq c \left(\frac{h_n^l}{h_n^j}\right)^{8/3k},
\end{split}
\end{equation*}
where $y'=\frac{h_n^l}{h_n^j}y+\frac{x_n^{l}-x_n^{j}}{h_n^j}$ and $s'=\left(\frac{h_n^l}{h_n^j}\right)^3s+\frac{t_n^{l}-t_n^{j}}{(h_n^j)^3}$. In the last line we have used the fact that $D_x^{2/3k}V^j$, $j\in\{1,\dots,J\}$ are continuous and compactly supported. By \eqref{hnhl}, the last term of the above inequality goes to zero as $n\rightarrow \infty$.

Next, assume that \eqref{hnxntn2} holds. Since $h^j_n=h^l_n$ for all $n\in \N$ in this case, the same computations yield
\begin{equation*}
\int \int  |D_x^{2/3k}V^j_n(x,t)||D_x^{2/3k}V^l_n(x,t)|^{3k/2-1}dxdt
\end{equation*}
\begin{equation*}
\begin{split}
&= \int \int  |D_x^{2/3k}V^j\left(y+\dfrac{x_n^{l}-x_n^{j}}{h_n^j}, s+\dfrac{t_n^{l}-t_n^{j}}{(h_n^j)^3} \right)||D_x^{2/3k}V^l(y,s)|^{3k/2-1}dyds.
\end{split}
\end{equation*}
Since $D_x^{2/3k}V^j$, $j\in\{1,\dots,J\}$ are continuous and compactly supported, in view of \eqref{hnxntn2}, we can apply the Lebesgue's Dominated Convergence, to conclude that the last line of the above inequality goes to zero as $n\rightarrow \infty$. Therefore,   we have
$$
\sum_{j\neq l}\int \int  |D_x^{2/3k}V^j_n||D_x^{2/3k}V^l_n|^{3k/2-1}dxdt \rightarrow 0,
$$
as $n\rightarrow \infty$, which leads to the desired claim \eqref{SPWNL9} if $D_x^{2/3k}V^j$, $j\in\{1,\dots,J\}$ are continuous and compactly supported

Finally, in the general case we use  a density argument as follows. Let $V^j_\eps$  be a function such that  $D_x^{2/3k}V^j_\eps$ is  continuous, compactly  supported and  converges to $D_x^{2/3k}V^j$ in  ${L^{3k/2}_{t,x}}$ when $\eps$ goes to zero. Then, from  (\ref{chain_ruleDx}) we conclude that
$$
\left\|D_x^{2/3k} \frac{1}{(h_n^j)^{2/k }}   V^j_\eps \left(\dfrac{x-x_n^j}{h_n^j}, \dfrac{t-t_n^j}{(h_n^j)^{3}}\right) -D_x^{2/3k}    V^j_n(x,t) \right\|^{3k/2}_{L^{3k/2}_{t,x}}
$$
\begin{equation}\label{conv_D_en_unif}
= \left\|D_x^{2/3k}V^j_\eps - D_x^{2/3k} V^j\right\|^{3k/2}_{L^{3k/2}_{t,x}}.
\end{equation}
%\begin{multline}\label{conv_D_en_unif}
%\left\|D_x^{2/3k} \frac{1}{(h_n^j)^{2/k }}   V^j_\eps \left(\dfrac{x-x_n^j}{h_n^j}, \dfrac{t-t_n^j}{(h_n^j)^{3}}\right) -D_x^{2/3k}    V^j_n(x,t) \right\|^{3k/2}_{L^{3k/2}_{t,x}}=\\
% \left\| \dfrac{1}{(h_n^j)^{2/k + 2/3k}}\left( (D_x^{2/3k}V^j_\eps)\left(\dfrac{x-x_n^j}{h_n^j}, \dfrac{t-t_n^j}{(h_n^j)^{3}}\right) - (D_x^{2/3k}V^j)\left(\dfrac{x-x_n^j}{h_n^j}, \dfrac{t-t_n^j}{(h_n^j)^{3}}\right) \right) \right\|^{3k/2}_{L^{3k/2}_{t,x}} \\
%= \left\|D_x^{2/3k}V^j_\eps - D_x^{2/3k} V^j\right\|^{3k/2}_{L^{3k/2}_{t,x}}.
%\end{multline} 
Hence,  the convergence when $\epsilon$ goes to zero of the first term on the left-hand side of the last equation is uniform in $n$. Also, since  $D_x^{2/3k}V^j_\eps$ is  continuous, and compactly  supported we have
$$
\limsup_{n\rightarrow \infty} \left\|D_x^{2/3k}\sum_{j=1}^{J} (h_n^j)^{2/k }   V^j_\eps \left(\dfrac{x-x_n^j}{h_n^j}, \dfrac{t-t_n^j}{(h_n^j)^{3}}\right) \right\|^{3k/2}_{L^{3k/2}_{t,x}} = \sum_{j=1}^{J}\left\|D_x^{2/3k} V^j_\eps \right\|^{3k/2}_{L^{3k/2}_{t,x}}.
$$
Finally, taking the limit $\eps \ra 0 $, interchanging the order of the limits and using (\ref{conv_D_en_unif}) we complete the proof of Theorem \ref{profdec}.

\end{proof}

\begin{remark}\label{Vnj}
 Let $V^j_n$ be defined by \eqref{Not1}, arguing similarly as in the end of the proof of Theorem \ref{profdec}, we can also prove that
\begin{equation}\label{VnjVnl}
\|D^{1/p}_xV^j_nD^{1/p}_xV^l_n\|_{L_t^{p/2}L_x^{q/2}} \rightarrow 0,
\end{equation}
as $n\rightarrow \infty$, for every $\dot{H}_x^{s_k}$-admissible pair $(p,q)$ satisfying \eqref{Adm}, with $k+2<p<\infty$.

Indeed, applying again the change of variables $x=h_n^ly+x_n^l$ and $t=(h_n^l)^3s+t_n^l$ we have
$$
\|D^{1/p}_xV^j_nD^{1/p}_xV^l_n\|_{L_t^{p/2}L_x^{q/2}} 
$$
\begin{eqnarray*}
&=& \left(\int \left(\int  |D^{1/p}_xV^j_n(x,t)D^{1/p}_xV^l_n(x,t)|^{q/2}dx\right)^{p/q}dt\right)^{2/p}\\
&=& \left(\frac{h_n^l}{h_n^j}\right)^{\frac{1}{q}+\frac{3}{p}} \left(\int \left(\int  |D^{1/p}_xV^j\left(y', s' \right)D^{1/p}_xV^l(y,s)|^{q/2}dy\right)^{p/q}ds\right)^{2/p}\!\!\!\!\!\!\!\!,
\end{eqnarray*}
where $y'=\frac{h_n^l}{h_n^j}y+\frac{x_n^{l}-x_n^{j}}{h_n^j}$ and $s'= \left(\frac{h_n^l}{h_n^j}\right)^3s+\frac{t_n^{l}-t_n^{j}}{(h_n^j)^3}$, which implies \eqref{VnjVnl}, if $((h_n^j)_{n\in \N}, (x_n^j)_{n\in \N},  (t_n^j)_{n\in \N})$ and $((h_n^l)_{n\in \N}, (x_n^l)_{n\in \N},  (t_n^l)_{n\in \N})$ are orthogonal.

\end{remark}

\section{Existence of Maximizers}\label{mainresult}
We start this section with the following auxiliary lemma.

\begin{lemma}\label{mainlemma}
 Let $(h_n^j, x_n^j, t_n^j)_{n\in \mathbb{N}, j\in \mathbb{N}}$ be a sequence  satisfying the condition \eqref{XT} and $(p,q)$ be  a pair satisfying \eqref{Adm}  with $k+2<p<\infty$. Then, for all $J\geq 1$,
 $$
 \limsup_{n\rightarrow \infty} \left\|D_x^{1/p}\left[\sum_{j=1}^{J}\dfrac{1}{(h_n^j)^{2/k}}U\left(\dfrac{t-t_n^j}{(h_n^j)^{3}}\right)\psi^j\left(\dfrac{x-x_n^j}{h_n^j}\right)\right]\right\|^B_{L^p_tL^{q}_x}
 $$
\begin{equation}\label{D1}
 \leq  \sum_{j=1}^{J}\left\|D_x^{1/p}U(t)\psi^j\right\|^B_{L^p_tL^{q}_x},
\end{equation}
where $B=\min\{p,q\}$.

Also,
$$
\limsup_{n\rightarrow \infty} \left\|\sum_{j=1}^{J}\dfrac{1}{(h_n^j)^{2/k}}U\left(\dfrac{t-t_n^j}{(h_n^j)^{3}}\right)\psi^j\left(\dfrac{x-x_n^j}{h_n^j}\right)\right\|^{5k/4}_{L^{5k/4}_xL^{5k/2}_t} 
$$
\begin{equation}\label{D2}
\leq  \sum_{j=1}^{J}\left\|U(t)\psi^j\right\|^{5k/4}_{L^{5k/4}_xL^{5k/2}_t}.
\end{equation}
\end{lemma}

 Lemma \ref{mainlemma} is similar to Lemma 5.5 from \cite{BeVa07}, Lemma 1.6 from \cite{Sh091},  Lemma 5.1 from  \cite{Sh092}, and Lemma 3.11 from \cite{Bu10}.  It  follows from Remark \ref{Vnj} and arguments similar to the ones used in the proofs of the just mentioned Lemmas, which we present here for the readers convenience.

\begin{proof}[Proof of Lemma \ref{mainlemma}]
We first observe that for $f, \; g \in L^{p/q}_xL^1_t$
$$
\|f+g\|^{B/q}_{L^{p/q}_tL^1_x} \leq \|f\|^{B/q}_{L^{p/q}_tL^1_x} +\|g\|^{B/q}_{L^{p/q}_tL^1_x}.
$$
Again, we use the notation \eqref{Not1} and \eqref{Not2}. Hence,
\begin{eqnarray*}
\left\|D^{1/p} \sum_{j=1}^J V_n^j\right\|_{L^{p}_tL^q_x}^B &=& \left\| \left|D_x^{1/p} \sum_{j=1}^J V_n^j\right|^q \right\|_{L^{p/q}_tL^1_x}^{B/q}\\
& =&
\left\| \left|D_x^{1/p} \sum_{j=1}^J V_n^j\right| \left|D_x^{1/p} \sum_{i=1}^J V_n^i\right| \left|D_x^{1/p} \sum_{k=1}^J V_n^k\right|^{q-2} \right\|_{L^{p/q}_tL^1_x}^{B/q}\\
&\leq & \left\|  \sum_{j=1}^J \left| D_x^{1/p} V_n^j\right|   \left(\sum_{i=1}^J \left| D_x^{1/p} V_n^i\right| \right)\left|D_x^{1/p}\sum_{k=1}^J V_n^k\right|^{q-2} \right\|_{L^{p/q}_tL^1_x}^{B/q} \\
&\leq & \sum_{j=1}^J \left[ \left\|    \left| D_x^{1/p} V_n^j\right|^2 \left|D_x^{1/p} \sum_{k=1}^J V_n^k\right|^{q-2} \right\|_{L^{p/q}_tL^1_x}^{B/q}\right.\\
&&+   \left. \left\| \sum_{i\neq j, i\leq J} \left| D_x^{1/p} V_n^i\right|\left| D_x^{1/p} V_n^j\right| \left|D_x^{1/p} \sum_{k=1}^J V_n^k\right|^{q-2}  \right\|_{L^{p/q}_tL^1_x}^{B/q} \right] \\
&=&\sum_{j=1}^J I_j+ II_j.
\end{eqnarray*}

We first consider the term $II_j$. Since $\frac{q}{p}=\frac{2}{p}+\frac{q-2}{p}$ and $1=\frac{2}{q}+\frac{q-2}{q}$, we apply H\"older's inequality to obtain
\begin{eqnarray}
II_{j}&\leq & \sum_{i\neq j, i\leq J}\left\|    \left| D_x^{1/p} V_n^j D_x^{1/p} V_n^i \right|\left| \sum_{k=1}^J D_x^{1/p} V_n^k \right|^{q-2} \right\|_{L^{p/q}_tL^{1}_x}^{B/q} \nonumber \\
&\leq & \sum_{i\neq j, i\leq J}\left\|     D_x^{1/p} V_n^j D_x^{1/p} V_n^i  \right\|_{L^{p/2}_tL^{q/2}_x}^{B/q}
\left\|    \left| \sum_{k=1}^J D_x^{1/p} V_n^k \right|^{q-2} \right\|_{L^{\frac{p}{q-2}}_tL^{\frac{q}{q-2}}_x}^{B/q} \nonumber \\
&\leq & \sum_{i\neq j, i\leq J}\left\|     D_x^{1/p} V_n^j D_x^{1/p} V_n^i  \right\|_{L^{p/2}_tL^{q/2}_x}^{B/q}
\left\|     \sum_{k=1}^J D_x^{1/p} V_n^k  \right\|_{L^{p}_tL^{q}_t}^{B(q-2)/r}. \label{d:LDes:eq_II}
\end{eqnarray}

Next, we study the term $I_j$. Set $s$ as the greatest integer less or equal than  $q-2$. Hence, $0\leq q-2-s<1$ and for $j=1,\dots, J$
\begin{eqnarray*}
 I_j &\leq &   \left\|    \left| D_x^{1/p} V_n^j\right|^2  \left(\sum_{i=1}^J \left| D_x^{1/p}  V_n^i\right|\right)^{q-2-s} \left|D_x^{1/p} \sum_{k=1}^J V_n^k\right|^{s} \right\|_{L^{p/q}_tL^1_x}^{B/q} \\
 & \leq &  \left\|    \left| D_x^{1/p} V_n^j\right|^2  \left(\sum_{i=1}^J \left| D_x^{1/p}  V_n^i\right|^{q-2-s} \right) \left|D_x^{1/p} \sum_{k=1}^J V_n^k\right|^{s} \right\|_{L^{p/q}_tL^1_x}^{B/q} \\
 & \leq & \left\|    \left| D_x^{1/p} V_n^j\right|^{q-s}  \left|D_x^{1/p} \sum_{k=1}^J V_n^k\right|^{s} \right\|_{L^{p/q}_tL^1_x}^{B/q} \\
 & & ~~~~~ +
 \sum_{i\neq j, i\leq J}^J \left\|    \left| D_x^{1/p} V_n^j\right|^{4+s-q}   \left|D_x^{1/p} V_n^j D_x^{1/p}  V_n^i\right|^{q-2-s}  \left|D_x^{1/p} \sum_{k=1}^J V_n^k\right|^{s} \right\|_{L^{p/q}_tL^1_x}^{B/q}.
\end{eqnarray*}
As a consequence,
\begin{eqnarray}
 I_j &\leq &  \!\! \left\|    \left| D_x^{1/p} V_n^j\right|^{q-s}  \left( \sum_{k=1}^J \left|D_x^{1/p}  V_n^k\right|^{s} + \!\!\!\! \sum_{(k_1,\dots,k_s)} \left|D_x^{1/p}  V_n^{k_1}\right| \ldots \left|D_x^{1/p}  V_n^{k_s}\right|
 \right) \right\|_{L^{p/q}_tL^1_x}^{B/q} \nonumber \\
 & & ~~~~~ +
 \sum_{i\neq j, i\leq J}^J \left\|    \left| D_x^{1/p} V_n^j\right|^{4+s-q}   \left|D_x^{1/p} V_n^j D_x^{1/p}  V_n^i\right|^{q-2-s}  \left|D_x^{1/p} \sum_{k=1}^J V_n^k\right|^{s} \right\|_{L^{p/q}_tL^1_x}^{B/q} \nonumber \\
 &\leq &  \sum_{k=1}^J \left\|    \left| D_x^{1/p} V_n^j\right|^{q-s}    \left|D_x^{1/p}  V_n^k\right|^{s} \right\|_{L^{p/q}_xL^1_t}^{B/q} \nonumber \\
 & & ~~~~~ + \sum_{(k_1,\ldots,k_s)} \left\| \left| D_x^{1/p} V_n^j\right|^{q-s}\left|D_x^{1/p}  V_n^{k_1}\right|\ldots \left|D_x^{1/p}  V_n^{k_s}\right|
 \right\|_{L^{p/q}_tL^1_x}^{B/q} \nonumber \\
 & & ~~~~~ +
 \sum_{i\neq j, i\leq J}^J \left\|    \left| D_x^{1/p} V_n^j\right|^{4+s-q}   \left|D_x^{1/p} V_n^j D_x^{1/p}  V_n^i\right|^{q-2-s}  \left|D_x^{1/p} \sum_{k=1}^J V_n^k\right|^{s} \right\|_{L^{p/q}_tL^1_x}^{B/q} \nonumber \\
&\leq  & \left\|\left| D_x^{1/p} V_n^j\right|^q\right\|_{L^{p/q}_xL^1_t}^{B/q}+ \sum_{k\neq j, k\leq J} {Ia}_{j,k} \nonumber \\
& & ~~~~~+\sum_{(k_1,\dots,k_s)} Ib_{j,k_1,\dots,k_s} +  \sum_{i\neq j, i\leq J} Ic_{j,i}. \label{d:LDes:eq_I_j}
\end{eqnarray}
where
\begin{equation*}
\begin{split}
{Ia}_{j,k} &= \left\|    \left| D_x^{1/p} V_n^j\right|^{q-s}   \left|D_x^{1/p} V_n^k \right|^{s} \right\|_{L^{p/q}_tL^1_x}^{B/q},\\
Ib_{j,k_1,\dots,k_s} &= \left\| \left| D_x^{1/p} V_n^j\right|^{q-s}\left|D_x^{1/p}  V_n^{k_1}\right|\ldots \left|D_x^{1/p}  V_n^{k_s}\right| \right\|_{L^{p/q}_tL^1_x}^{B/q},\\
Ic_{j,i} &= \left\|    \left| D_x^{1/p} V_n^j\right|^{4+s-q}   \left|D_x^{1/p} V_n^j D_x^{1/p}  V_n^i\right|^{q-2-s}  \left|D_x^{1/p} \sum_{k=1}^J V_n^k\right|^{s} \right\|_{L^{p/q}_tL^1_x}^{B/q}.
\end{split}
\end{equation*}

First note that a simple computation yields
\begin{equation}\label{Vjn2}
\left\|\left| D_x^{1/p} V_n^j\right|^q\right\|_{L^{p/q}_xL^1_t}^{B/q}=\left\|D_x^{1/p}V^j\right\|^B_{L^p_tL^{q}_x}.
\end{equation}

Next, we observe that if $q-s\leq s$, then for $j=1,\dots,J$, $k\neq j$ we use H\"older's inequality to obtain
\begin{eqnarray}
Ia_{j,k}&=& \left\|    \left| D_x^{1/p} V_n^j\right|^{q-s} \left|D_x^{1/p} V_n^k \right|^{q-s}  \left|D_x^{1/p} V_n^k \right|^{2s-q} \right\|_{L^{p/q}_tL^1_x}^{B/q} \nonumber \\
& \leq & \left\|  \left| D_x^{1/p} V_n^j D_x^{1/p} V_n^k \right|^{q-s}
\right\|_{L^{\frac{p}{2q-2s}}_t L^{\frac{q}{2q-2s}}_x}^{B/q}
\left\|  \left| D_x^{1/p} V_n^k \right|^{2s-q}
\right\|_{L^{\frac{p}{2s-q}}_t L^{\frac{q}{2s-q}}_x}^{B/q} \nonumber \\
& \leq & \left\|   D_x^{1/p} V_n^j D_x^{1/p} V_n^k
\right\|_{L^{\frac{p}{2}}_t L^{\frac{q}{2}}_x}^{B(q-s)/q}
\left\|   D_x^{1/p} V_n^k \right\|_{L^{p}_t L^{q}_t}^{B(2s-q)/q}. \label{d:LDes:eq_Ia_1}
\end{eqnarray}

On the other hand, if $q-s>s$ another application of the H\"older's inequality yields
\begin{eqnarray}
Ia_{j,k}&=& \left\|    \left| D_x^{1/p} V_n^j\right|^{s} \left|D_x^{1/p} V_n^k \right|^{s}  \left|D_x^{1/p} V_n^k \right|^{q-2s} \right\|_{L^{p/q}_xL^1_t}^{B/q} \nonumber  \\
& \leq & \left\|  \left| D_x^{1/p} V_n^j D_x^{1/p} V_n^k \right|^{s}
\right\|_{L^{\frac{p}{2s}}_t L^{\frac{q}{2s}}_x}^{B/q}
\left\|  \left| D_x^{1/p} V_n^k \right|^{q-2s}
\right\|_{L^{\frac{p}{q-2s}}_t L^{\frac{q}{q-2s}}_x}^{B/q} \nonumber \\
& \leq & \left\|   D_x^{1/p} V_n^j D_x^{1/p} V_n^k
\right\|_{L^{\frac{p}{2}}_t L^{\frac{q}{2}}_x}^{Bs/q}
\left\|   D_x^{1/p} V_n^k \right\|_{L^{p}_t L^{q}_t}^{B(q-2s)/r}. \label{d:LDes:eq_Ia_2}
\end{eqnarray}

Now, note that the sum of the terms $Ib$ is over the $s$-tuples $(k_1,\dots, k_s)$ such that at least two $k_i's$ are different (if $s$ are equal to 0 or 1 the result is trivial). Without loss of generality we assume $k_1\neq k_2$.  Since $q/p=2/p+(q-s)/p+(s-2)/p$ and $1=2/q+(q-s)/q+(s-2)/q$, for $j \in\{1,\dots,J\}$ and $(k_1,\dots,k_s)$ with $k_1\neq k_2$, we have by H\"older's inequality
\begin{eqnarray}
Ib_{j,k_1,\dots,k_s}&\leq & \left(\left\| D_x^{1/p} V_n^j \right\|^{q-s}_{L^{p}_tL^{q}_x}
\left\| D_x^{1/p}  V_n^{k_1} D_x^{1/p}  V_n^{k_2} \right\|_{L^{p/2}_tL^{q/2}_x}\right. \nonumber  \\
& & \quad \left. \left\| D_x^{1/p}  V_n^{k_3} \right\|_{L^{p}_tL^{q}_t}   \ldots  \left\|  D_x^{1/p}  V_n^{k_s} \right
\|_{L^{p}_t L^{q}_t} \right)^{B/q} \label{d:LDes:eq_Ib}
\end{eqnarray}

Finally, for $j\in\{1,\dots,J\}$, $i\neq j$ we deduce
\begin{eqnarray}
Ic_{j,i}&\leq & \left\|    \left| D_x^{1/p} V_n^j D_x^{1/p} V_n^i \right|^{q-2-s} \right\|_{L^{\frac{p}{2q-4-2s}}_tL^{\frac{q}{2q-4-2s}}_x}^{B/q} \nonumber \\
&&\quad \left\|    \left| D_x^{1/p} V_n^j \right|^{4+s-q} \left|\sum_{k=1}^J D_x^{1/p} V_n^k \right|^s \right\|_{L^{\frac{p}{4+2s-q}}_tL^{\frac{q}{4+2s-q}}_x}^{B/q} \nonumber \\
&\leq & \left\|     D_x^{1/p} V_n^j D_x^{1/p} V_n^i    \right\|_{L^{\frac{p}{2}}_tL^{\frac{q}{2}}_x}^{B(q-2-s)/q}
\left\| \left| D_x^{1/p} V_n^j  \right|^{4+s-q}   \right\|_{L^{\frac{p}{4+s-q}}_tL^{\frac{q}{4+s-q}}_x}^{B/q}\nonumber \\
&&\quad \left\| \left|\sum_{k=1}^J D_x^{1/p} V_n^k \right|^s \right\|_{L^{p/s}_tL^{q/s}_x}^{B/q} \nonumber \\
 &\leq & \left\|     D_x^{1/p} V_n^j D_x^{1/p} V_n^i    \right\|_{L^{\frac{p}{2}}_tL^{\frac{q}{2}}_x}^{B(q-2-s)/q}
\left\|  D_x^{1/p} V_n^j     \right\|_{L^{P}_tL^{q}_t}^{B(4+s-q)/q}\nonumber \\
 &&\quad \left\| \sum_{k=1}^J D_x^{1/p} V_n^k   \right\|_{L^{P}_tL^{q}_t}^{Bs/q}.\label{d:LDes:eq_Ic}
\end{eqnarray}

Taking the limit $n \rightarrow \infty$ in inequalities \eqref{d:LDes:eq_II}, \eqref{d:LDes:eq_Ia_1}-\eqref{d:LDes:eq_Ic}, using Remark \ref{Vnj}, identity \eqref{Vjn2}, the Strichartz inequality (\ref{STR0}), and the fact that  $\psi^j$ is bounded in $\dot{H}^{s_k}$ by \eqref{PYTHA}, we obtain (\ref{D1}). The above  arguments with straightforward modifications, and inequality (\ref{STR3}) yield (\ref{D2}).
\end{proof}

Next, we obtain the existence of maximizers for the Airy-Strichartz inequalities \eqref{STR0} and \eqref{STR3}, respectively.

\begin{proof}[Proof of Theorem  \ref{Max1}]
For a fixed $k>4$,  let $(u_n)_{n\in \mathbb{N}}\subset \dot{H}_x^{s_k}$ be a sequence satisfying $\|D_x^{s_k}u_n\|_{L^2_x}=1$ and
$$
\lim_{n\rightarrow \infty} \|D^{1/p}_xU(t)u_n \|_{L_t^pL_x^q} = M_k.
$$
Let
$(\psi^j)_{j\in \mathbb{N}}\subset \dot{H}_x^{s_k}$,
$(v_n)_{n\in \mathbb{N}} \subset L_t^pL_x^q $  and
$(h^j_n,x^j_n,t^j_n)_{j\in \mathbb{N}}$ be sequences associated to $(u_n)_{n\in \mathbb{N}}$ according to Theorem \ref{profdec}. Hence, given $\epsilon>0 $
there exists $N_0 \in \mathbb{N}$ such that for $n, J>N_0$
\begin{equation}\label{d:teo:max1:eq1}
M_k-\|D^{1/p}_xU(t)u_n \|_{L_t^pL_x^q} \leq \eps/2, ~\mbox{  }~ \limsup_{n\rightarrow \infty}\|D^{1/p}_xU(t)R^J_n \|_{L_t^pL_x^q} \leq \frac{\eps}{4},
\end{equation}
and  from \eqref{PYTHA}
\begin{equation}\label{d:teo:max1:Lpit}
\sum_{j=1}^{J}  \|\psi^{j}\|_{\dot{H}^{s_k}_x}^2 \leq \|u_n \|_{\dot{H}^{s_k}_x}^2 +\eps \leq 1+\eps.
\end{equation}
Next, recalling the notation \eqref{Not1} and \eqref{Not2}, we observe that for each $J>N_0$ there exists  $N_J  \in \mathbb{N}$ such that if $n>N_J$
\begin{equation}\label{d:teo:max1:eq2}
\|D^{1/p}_xU(t)R^J_n \|_{L_t^pL_x^q} \leq \frac{\eps}{2}
\end{equation}
\begin{equation}\label{d:teo:max1:eq3}
\left\|D_x^{1/p}\left[\sum_{j=1}^{J}V^j_n\right]\right\|^B_{L^p_tL^{q}_x} \leq  \sum_{j=1}^{J}\left\|D_x^{1/p}V^j\right\|^B_{L^p_tL^{q}_x} +\eps.
\end{equation}
Here we have used Lemma \ref{mainlemma} to obtain the last inequality, and $B=\min\{p,q\}$.  Hence,  if $J>N_0$ and $n>N_J$
\begin{eqnarray}
\sum_{j=1}^{J}\left\|D_x^{1/p}V^j\right\|^B_{L^p_tL^{q}_x} &\geq& \left\|D_x^{1/p}\left[\sum_{j=1}^{J}V^j_n\right]\right\|^B_{L^p_tL^{q}_x} -\eps \nonumber \\
& \geq & \left\| D_x^{1/p}U(t)u_n -  D_x^{1/p}U(t)R_n^J(x) \right\|^B_{L^p_tL^{q}_x} -\eps \nonumber \\
& \geq & \left| \left\|D_x^{1/p} U(t)u_n \right\|_{L^p_tL^{q}_x} - \left\| D_x^{1/p} U(t) R_n^J(x) \right\|_{L^p_tL^{q}_x} \right|^B -\eps \nonumber\\
& \geq & ( M_k - \eps )^B -\eps \label{d:teo:max:sum_M_k}
\end{eqnarray}
where we have used  inequalities (\ref{d:teo:max1:eq1})-(\ref{d:teo:max1:eq3}) to obtain the last estimate.

Now, set $j^*\in \{1,..,J\}$ as the smallest index such that 
$$
\|D_x^{1/p}V^{j^*}\|_{L^p_tL^{q}_x}= \max_{1\leq j \leq J}\{\|D_x^{1/p} V^j\|_{L^p_tL^{q}_x}\}.
$$ 
Hence, the
 last inequality, and Strichartz estimate  (\ref{STR0}) yield
\begin{eqnarray}
( M_k - \eps )^B -\eps &\leq & \left\|D_x^{1/p}V^{j^*}\right\|^{B-2}_{L^p_tL^{q}_x}\sum_{j=1}^{J}\left\|D_x^{1/p}V^j\right\|^2_{L^p_tL^{q}_x} \nonumber \\
& \leq & M_k^{B-2} \left( \|\psi^{j^*}\|^2_{\dot{H}^{s_k}_x} \right)^{\frac{B-2}{2}}\sum_{j=1}^{J} \left( M_k \|\psi^{j}\|_{\dot{H}^{s_k}_x} \right)^{2} \nonumber\\
& \leq & M_k^{B-2} \left( \|{\psi^{j^*}}\|^2_{\dot{H}^{s_k}_x} \right)^{\frac{B-2}{2}}
M_k^2  (1+\eps)\nonumber\\
& \leq & M_k^B (1+\eps)^{B/2}. \label{d:teo:max:sand}
\end{eqnarray}
Here we  used  (\ref{d:teo:max1:Lpit}) to obtain the last two inequalities.

 We now claim that $j^*$ is independent of $J$ for sufficiently large $J$ (or  equivalently $\epsilon$ sufficiently  small). Indeed, from (\ref{d:teo:max1:Lpit}) we have
 $$
 \lim_{j\rightarrow \infty }\|\psi^{j}\|_{\dot{H}^{s_k}_x}^2 =0.
 $$
Hence,  from (\ref{STR0}) and the definition of  $V^j$ we obtain
$$
\lim_{j\rightarrow \infty }\left\|D_x^{1/p}V^j\right\|_{L^p_tL^{q}_x} =0,
$$
allowing us to conclude our claim. As a consequence, from (\ref{d:teo:max:sand}) we conclude that for   sufficiently  small $\epsilon$ we have
$$
( M_k - \eps )^B -\eps \leq
M_k^{B} \left( \|\psi^{j^*}\|^2_{\dot{H}^{s_k}_x} \right)^{\frac{B-2}{2}}
  (1+\eps)  \leq  M_k^B (1+\eps)^{B/2}
$$
with $j^*$ independent of $J$ and $\eps$. Hence,  taking $\eps \rightarrow 0$ yields
$$
\|{\psi^{j^*}}\|^2_{\dot{H}^{s_k}_x}=1
$$
 and from (\ref{d:teo:max1:Lpit})  for $j\neq j^* $  we conclude that
$$
\|\psi^{j}\|^2_{\dot{H}^{s_k}_x}=0.
$$
Finally, we observe from (\ref{d:teo:max:sum_M_k}) and Strichartz estimate  (\ref{STR0})  that for sufficiently small $\eps$
$$
 ( M_k - \eps )^B -\eps \leq \sum_{j=1}^{J}\left\|D_x^{1/p}V^j\right\|^B_{L^p_tL^{q}_x} = \left\|D_x^{1/p}V^{j^*}\right\|^B_{L^p_tL^{q}_x}.
$$
Taking $\epsilon \rightarrow 0 $ we conclude that
$$
\left\|D_x^{1/p}U(t)\psi^{j^*}\right\|_{L^p_tL^{q}_x}=  \left\|D_x^{1/p}V^{j^*}\right\|_{L^p_tL^{q}_x} =M_k = M_k \|\psi^{j^*}\|_{\dot{H}^{s_k}_x}.
$$
\end{proof}

\begin{proof}[Proof of Theorem  \ref{Max2}]
This result follows from arguments similar to the ones used in the  proof of Theorem \ref{Max1} and inequalities (\ref{STR3}) and (\ref{D2}).

\end{proof}

\vspace{0.5cm}
\centerline{\textbf{Acknowledgments}}
\vspace{0.5cm}
L. G. F. was partially supported by CNPq/Brazil and FAPEMIG/Brazil. The authors would like to thank the referee for the careful review and the valuable comments which helped to improve the manuscript.\\

%\bibliographystyle{abbrv}
%\bibliography{bibfarah-2014-09-09}

\end{document}